\documentclass[11pt]{amsart}
\usepackage[T1]{fontenc} 
\usepackage{amsmath,amsthm,amsfonts,amssymb,bm,graphicx,xcolor}

\usepackage
{hyperref}
\hypersetup{colorlinks=true,citecolor=blue,linkcolor=blue,urlcolor=blue,
pdfstartview=FitH }

\theoremstyle{plain}
\newtheorem{theorem}{Theorem}
\newtheorem{cor}{Corollary}
\newtheorem{lemma}{Lemma}
\newtheorem{prop}[cor]{Proposition}

\theoremstyle{definition}

\renewcommand{\geq}{\geqslant}
\renewcommand{\leq}{\le}
\renewcommand{\mod}{\mathrm{mod}\,}

\newcommand{\Z}{\mathbb{Z}}
\newcommand{\Q}{\mathbb{Q}}
\newcommand{\N}{\mathbb{N}}

\newcommand{\G}{\mathbb{G}}

\newcommand{\bs}{\backslash}
\newcommand{\cb}[1]{\left\{{#1}\right\}}
\newcommand{\cbm}[2]{\left\{{#1}\; \middle|\; {#2}\right\}}

\newcommand{\height}{\operatorname{ht}}

\newcommand{\ol}[1]{\overline{#1}}

\newcommand{\rb}[1]{\left({#1}\right)}

\newcommand{\sbe}{\subseteq}

\newcommand{\ul}[1]{\underline{#1}}

\newcommand{\Hom}{\operatorname{Hom}}

\newcommand{\GL}{\operatorname{GL}}

\newcommand{\SL}{\operatorname{SL}}

\newcommand{\hra}{\hookrightarrow}

\newcommand{\mc}{\mathcal}
\newcommand{\mf}{\mathfrak}

\newcommand{\bdd}{\begin{center}\begin{tikzcd}}
\newcommand{\bd}{\begin{tikzcd}}
\newcommand{\edd}{\end{tikzcd}\end{center}}
\newcommand{\ed}{\end{tikzcd}}
\newcommand{\bdp}{\begin{center}\begin{tikzpicture}}
\newcommand{\edp}{\end{tikzpicture}\end{center}}
\newcommand{\bi}{\begin{itemize}}
\newcommand{\ei}{\end{itemize}}
\newcommand{\bt}{\begin{tikzpicture}}
\newcommand{\et}{\end{tikzpicture}}
\newcommand{\ba}{\[\begin{aligned}}
\newcommand{\ea}{\end{aligned}\]}
\newcommand{\bp}{\begin{pmatrix}}
\newcommand{\ep}{\end{pmatrix}}
\newcommand{\bv}{\begin{vmatrix}}
\newcommand{\ev}{\end{vmatrix}}
\newcommand{\bb}{\begin{bmatrix}}
\newcommand{\eb}{\end{bmatrix}}
\newcommand{\bB}{\begin{Bmatrix}}
\newcommand{\eB}{\end{Bmatrix}}
\newcommand{\bea}{\begin{enumerate}[label=(\alph*)]}
\newcommand{\ber}{\begin{enumerate}[label=(\roman*)]}
\newcommand{\ben}{\begin{enumerate}[label=(\arabic*)]}
\newcommand{\ee}{\end{enumerate}}

\numberwithin{equation}{section}

\makeatletter
\def\Ddots{\mathinner{\mkern1mu\raise\p@
\vbox{\kern7\p@\hbox{.}}\mkern2mu
\raise4\p@\hbox{.}\mkern2mu\raise7\p@\hbox{.}\mkern1mu}}
\makeatother
\makeatletter
\DeclareRobustCommand\widecheck[1]{{\mathpalette\@widecheck{#1}}}
\def\@widecheck#1#2{%
    \setbox\z@\hbox{\m@th$#1#2$}%
    \setbox\tw@\hbox{\m@th$#1%
       \widehat{%
          \vrule\@width\z@\@height\ht\z@
          \vrule\@height\z@\@width\wd\z@}$}%
    \dp\tw@-\ht\z@
    \@tempdima\ht\z@ \advance\@tempdima2\ht\tw@ \divide\@tempdima\thr@@
    \setbox\tw@\hbox{%
       \raise\@tempdima\hbox{\scalebox{1}[-1]{\lower\@tempdima\box
\tw@}}}%
    {\ooalign{\box\tw@ \cr \box\z@}}}
\makeatother

\begin{document}

\author{Valentin Blomer}
\address{Universit\"at Bonn, Mathematisches Institut, Endenicher Allee 60, 53115 Bonn, Germany}
\email{blomer@math.uni-bonn.de}
 
\author{Siu Hang Man}
\address{Charles University, Faculty of Mathematics and Physics, Department of Algebra, Sokolovsk\' a 83, 186~75 Praha~8, Czechia}
\email{shman@karlin.mff.cuni.cz}

 \title{Bounds for Kloosterman sums on ${\rm GL}(n)$}

\thanks{The first author was supported in part by the DFG-SNF Lead Agency Program grant BL 915/2-2,  Germany's Excellence Strategy grant EXC-2047/1 - 390685813 and ERC
Advanced Grant 101054336. The second author was supported by the European Research Council grant 101001179 and the Czech Science Foundation GA\v CR grant 21-00420M}


\begin{abstract}   This paper establishes  power-saving bounds for Kloosterman sums associated with the long Weyl element for ${\rm GL}(n)$ for arbitrary $n \geq 3$, as well as for another type of Weyl element of order 2.  These bounds are obtained by establishing an explicit representation as exponential sums. As an application we go beyond Sarnak's density conjecture for the principal congruence subgroup of prime level. We also obtain power-saving bounds for all Kloosterman sums on ${\rm GL}(4)$. 
 \end{abstract}


\setcounter{tocdepth}{2}  \maketitle 

\maketitle

\section{Introduction}

\subsection{Classical Kloosterman sums} Kloosterman sums belong the most universal exponential sums in number theory, algebra and automorphic forms. The classical Kloosterman sum for parameters $m, m'\in \Bbb{Z}$ and a modulus $c \in \Bbb{N}$ is
\begin{equation}\label{1}
S(m, m', c) =\underset{d\, (\text{mod }c)}{\left.\sum \right.^{\ast}} e\Big(\frac{md + m'\bar{d}}{c}\Big)
\end{equation}
where $d\bar{d} \equiv 1$ (mod $c$) and the asterisk indicates that the sum is over $(d, c) = 1$. Kloosterman \cite{Kl} introduced this type of exponential sum in his thesis as a crucial tool to study the representation of integers by positive quaternary quadratic forms.  Since then they have been ubiquitous in number theory, for instance as the finite Fourier transform of exponential sums containing modular inverses, as Fourier coefficients of classical Poincar\'e series, in various instances of delta-symbol methods and perhaps most prominently in the relative trace formula of Petersson-Kuznetsov type (to which Poincar\'e series are a precursor). One of their key properties is the Weil bound \cite{We}  (complemented by Sali\'e for powerful moduli \cite{Sa})
\begin{equation}\label{weil}
|S(m, m', c)| \leq \tau(c) c^{1/2} (m, m', c)^{1/2}
\end{equation}
where $\tau$ denotes the divisor function. This essentially exhibits square root cancellation relative to the trivial bound $|S(m, m', c)| \leq \phi(c) $ where $\phi$ is Euler's $\phi$-function. \\

We  proceed to describe the more general set-up of Kloosterman sums. Let $G$ be a reductive group, $T$ the maximal torus, and $U$ the standard maximal unipotent subgroup. Let $N = N_G(T)$ be the normalizer of $T$ in $G$, $W = N/T$ the Weyl group and $\omega : N \rightarrow W$ the quotient map. For $w\in W$, we define $U_w = U \cap w^{-1} U^{\top} w$ and $\bar{U}_w = U \cap w^{-1} U w$. For $n \in N(\Bbb{Q}_p)$ and a finite index subgroup $\Gamma \subseteq G(\Bbb{Z}_p)$  we define
$$C(n) = U(\Bbb{Q}_p) nU(\Bbb{Q}_p) \cap \Gamma, \quad  X(n) = U(\Bbb{Z}_p)\backslash C(n)/U_{\omega(n)}(\Bbb{Z}_p)$$
and the projection maps  
$$u : X(n) \rightarrow U(\Bbb{Z}_p)\backslash U(\Bbb{Q}_p), \quad u' : X(n) \rightarrow U_{\omega(n)}(\Bbb{Q}_p)/U_{\omega(n)}(\Bbb{Z}_p). $$
Let $\psi, \psi'$ be two characters on $U(\Bbb{Q}_p)$ that are trivial on $U(\Bbb{Z}_p)$. The (local) Kloosterman sum is then defined to be 
\begin{equation}\label{general}
\text{Kl}_p(\psi, \psi', n) = \sum_{x = u(x) n u'(x) \in X(n)} \psi(u(x)) \psi'(u'(x)).
\end{equation}
In this paper we will take $\Gamma = G(\Bbb{Z}_p)$, but we may think of $\Gamma$ as a   more general ``congruence subgroup''. Of course, by the Chinese remainder theorem it suffices to study local Kloosterman sums. There is some flexibility in the definition. For instance, some authors have $U_{\omega(n)}$ on the left and $U$ on the right.  \\

\textbf{Example 1:} For $G = {\rm GL}(2)$, $n = \left(\begin{smallmatrix} & -1/c\\ c & \end{smallmatrix}\right)$ (so that $U_{\omega(n)} = U)$ with $c = p^r$, say, $\psi\left(\left(\begin{smallmatrix} 1 & x\\ & 1\end{smallmatrix}\right)\right) = \psi_0(mx)$, $\psi'\left(\left(\begin{smallmatrix} 1 & y\\ & 1\end{smallmatrix}\right)\right) = \psi_0(m'y)$ for the standard additive character $\psi_0$ on $\Bbb{Q}_p/\Bbb{Z}_p$ and $m, m' \in \Bbb{Z}$ we have
$$\left(\begin{matrix}1 & x\\ & 1\end{matrix}\right)n  \left(\begin{matrix}1 & y\\ & 1\end{matrix}\right) = \left(\begin{matrix}cx & -1/c + cxy\\ c& cy\end{matrix}\right) \in {\rm GL}_2(\Bbb{Z}_p)$$
 if and only if $x, y \in  p^{-r}\Bbb{Z}_p/\Bbb{Z}_p$, $xy - p^{-2r} \in p^{-r}\Bbb{Z}_p$, and we recover \eqref{1}. \\
 
 \textbf{Example 2:} For $G = {\rm GL}(3)$ and $n = \left(\begin{smallmatrix} && -1/c_1\\ &c_1/c_2 & \\ c_2&&\end{smallmatrix}\right)$ in the big Bruhat cell  an explicit form of the corresponding Kloosterman sum becomes already much more complicated. Using Pl\"ucker coordinates, Bump, Friedberg and Goldfeld \cite[Section 4]{BFG} derived for two general characters $\psi_{n_1, n_2}$ and $\psi_{m_1, m_2}$ of $3$-by-$3$ upper triangular unipotent matrices the explicit formula for the global Kloosterman sum
 $$\sum_{\substack{B_1, C_1 \, (\text{mod }c_1)\\ B_2, C_2\, (\text{mod } c_2)\\(B_j, C_j, c_j) =1\\ c_1c_2 \mid c_1C_2 + B_1B_2 + C_1c_2 }}  e\Big(\frac{m_1B_1 + n_1(Y_1c_2 - Z_1B_2)}{c_1} + \frac{m_2B_2 + n_2(Y_2c_1 - Z_2B_1)}{c_2}\Big)$$
where $Y_1, Y_2, Z_1, Z_2$ are chosen such that $Y_jB_j + Z_jC_j \equiv 1\, (\text{mod } c_j)$ for $j = 1, 2$.\\

We see from these examples that Kloosterman sums, while defined rather naturally in terms of the Bruhat decomposition, turn out to be extremely complicated exponential sums. We are well equipped with deep technology from algebraic geometry and $p$-adic analysis to bound general exponential sums in a relatively sharp fashion, but it is unfortunately not clear a priori how to apply these to general Kloosterman sums given by \eqref{general}. In fact, even determining the size of $X(n)$ and hence the ``trivial'' bound for \eqref{general} (ignoring cancellation in the characters) is a deep result of D{\k a}browski and Reeder \cite{DR}, see Lemma \ref{lem:repcount} below. 

One case is classical: for $G = {\rm GL}(n)$ and the Weyl element $\left(\begin{smallmatrix} & 1 \\ I_{n-1} & \end{smallmatrix}\right)$ (``Voronoi element'') the Kloosterman sum \eqref{general} becomes a hyper-Kloosterman sum \cite[Theorem B]{Fr} in the sense of Deligne and hence its size is well-understood. In addition we have good bounds for Kloosterman sums on ${\rm GL}(3)$ \cite{St, BFG, DF}, ${\rm Sp}(4)$ \cite{Ma} and for the long Weyl element on ${\rm GL}(4)$  \cite[Appendix]{GSW}. In all other cases, the best we know is the ``trivial'' bound of D{\k a}browski-Reeder \cite{DR}. A typical strategy,  employed for instance in \cite{BFG, GSW, Br}, is to use Pl\"ucker coordinates to understand the Bruhat decomposition of ${\rm GL}(n)$ in an explicit fashion. Unfortunately, for large $n$ the Pl\"ucker relations become extremely complicated, and it seems that this path is not very promising in general. 

In this paper, we investigate Kloosterman sums for the general linear group $G = {\rm GL}(n+1)$ for \emph{arbitrary} $n \geq 2$. (The notation $n+1$ in place of $n$ is slightly more convenient.) We parametrize the unipotent upper triangular matrices $U$  of dimension $n+1$ as 
$$u = \left(\begin{smallmatrix} 1 &u_{11} & u_{12} & \cdots & u_{1n}\\ & 1 & u_{22} & \ldots & u_{2n}\\   && \ddots&& \vdots\\ &&&1 & u_{nn}\\ &&&& 1\end{smallmatrix}\right).$$
We consider two characters  $\psi, \psi'$ of $U(\Q_p)/U(\Bbb{Z}_p)$, defined by
\begin{equation}\label{char}
\psi(u) = e\Big(\sum_{j=1}^n \psi_j u_{jj}\Big), \quad \psi'(u) = e\Big(\sum_{j=1}^n \psi'_j u_{jj}\Big)
\end{equation}
for $\psi_j, \psi'_j \in \Bbb{Z}$, $1\le j \le n$. 
The Weyl group of $G$ consists of $(n+1)!$ permutation matrices, but only Weyl elements of the form
\begin{equation*} 
w = \left( \begin{smallmatrix} &  & & I_{d_1}\\  & & I_{d_2} &\\ & \Ddots & &\\ I_{d_r} & && \end{smallmatrix}\right)
\end{equation*}
with identity matrices $I_{d_j}$ of dimension $d_j$ lead to well-defined objects (see \cite[p.\ 175]{Fr}), which reduces the number of relevant Weyl elements to $2^n$. We consider two of them, namely the long Weyl element
$$w_l =   \left( \begin{smallmatrix} &&&&& \pm 1\\  &&&& \Ddots & \\  &&& - 1\\ &&1\\ &-1\\ 1 \end{smallmatrix}\right),$$
and with a particular application in mind (to be described in Subsection \ref{appl}) 
$$w_{\ast} =   \left( \begin{smallmatrix} &  & \pm 1\\    & -I_{n-1}& \\ 1 & & \end{smallmatrix}\right).$$
The chosen representatives of the Weyl elements $w_l$ and $w_{\ast}$ here satisfy that the determinant of the $(k\times k)$-minors formed by the bottom $k$ rows and appropriate columns have positive determinant. The actual choice of the representatives is unimportant to the theory, but this particular choice makes the computations in later sections more convenient.

We write a typical modulus as $C = (p^{r_1}, \ldots, p^{r_n})$ with $r_j \in \Bbb{N}_0$ which we embed into the torus $T$ as
$C^{\ast} = \text{diag}(p^{-r_1}, p^{r_1 - r_{2}}, \ldots, p^{r_{n-1} - r_n}, p^{r_n})$. The first main achievement of this paper in an explicit and reasonably compact expression of the considered Kloosterman sums as  exponential sums which we are going to describe in the following two subsections. We try to present this in a user-friendly way, which may ease further investigations.

\subsection{Kloosterman sums for $G = {\rm GL}(n+1)$ for the long Weyl element} Our stratification starts with the following decomposition. Let a modulus $C$ be given as above with an exponent vector $r = (r_1, \ldots, r_n)$. Let
\begin{equation}\label{mwl}
\mc M_{w_l}(r) := \Big\{\ul m = (m_{ij})_{1\le i \le j\le n} \in \N_0^{n(n+1)/2} \mid \sum_{i\le k \le j} m_{ij} = r_k, \; 1\le k \le n\Big\}.
\end{equation}
For $\ul m \in \mc M_{w_l}(r)$ we define
$$
\mc C_{w_l}(\ul m) := \Big\{\ul c = (c_{ij})_{1\le i \le j \le n} \mid c_{ij} \in \Bbb{Z}/{\textstyle \rb{\prod_{k=j}^n p^{m_{ik}}}} \Bbb{Z}, \; (c_{ij},p^{m_{ij}}) = 1\Big\}
$$
as well as the partial Kloosterman sum
${\rm Kl}_{p}(\underline{m}, \psi, \psi', w_l)$ as 
\begin{equation}\label{partial}
\begin{split}
&\sum_{\underline{c} \in \mathcal{C}_{w_l}(\underline{m})}  e\Bigg(\sum_{j=1}^n \sum_{i=1}^j\psi_j c_{i(j-1)} \overline{c_{ij}} \prod_{k=1}^{i-1} p^{m_{k(j-1)}} \prod_{k=1}^i p^{-m_{kj}}\\
&  +  \sum_{i=1}^n \sum_{j=1}^i\psi'_i c_{(n+1-i)(n+1-j)} \overline{c_{(n+2-i)(n+1-j)}} \prod_{k=1}^{j-1} p^{m_{(n+2-i)(n+1-k)}} \prod_{k=1}^j  p^{-m_{(n+1-i)(n+1-k)}}\Bigg).
\end{split}
\end{equation}
Here we employ the convention that  $c_{ij} = 1$ and $m_{ij} = 0$ for $i>j$. Moreover, for $1 \leq i \leq j \leq n$ we  set $\overline{c_{ij}} = 0$ when $m_{ij}=0$, and of course a bar means the modular inverse. It is not obvious from the definition, but will follow from the proof, that this expression is well-defined (i.e.\ independent of the system of representatives chosen for the $c_{ij}$). 

As an example, the case $n=3$ (i.e.\ $G={\rm GL}(4)$) is spelled out explicitly in \eqref{81} in the appendix, where in addition all Weyl elements for ${\rm GL}(4)$ are analyzed. 


\medskip

 We are now ready to state our first main result. 
\begin{theorem}\label{thm1} For $C = (p^{r_1}, \ldots, p^{r_n})$ we have with the above notation
$${\rm Kl}_{p}(  \psi, \psi', C^{\ast}w_l) = \sum_{\underline{m} \in \mathcal{M}_{w_l}(r)} {\rm Kl}_{p}(\underline{m}, \psi, \psi', w_l)$$
where ${\rm Kl}_{p}(\underline{m}, \psi, \psi', w_l)$ is defined in \eqref{partial}.
\end{theorem}

\textbf{Remarks:} The exact formula of Theorem \ref{thm1} has a number of nice features. 

First of all, we see the ``trivial bound'' of D{\k a}browski-Reeder with bare eyes by simply counting the number of terms in the summation:
\begin{equation}\label{jminusi}
\#\mathcal{C}_{w_l}(\underline{m}) \leq \prod_{1 \leq i \leq j \leq n} \prod_{k = j}^n p^{m_{ik}} = \prod_{1 \leq i \leq j \leq n} p^{(j-i+1)m_{ij}} = \prod_{1 \leq k \leq n} p^{r_k}
\end{equation}
for $\underline{m} \in \mathcal{M}_{w_l}(r)$ as defined in \eqref{mwl}. We are now in a position to exploit cancellation to obtain non-trivial bounds in Corollary \ref{cor3} below. 

 It is also structurally quite interesting because it consists of nested ${\rm GL}(2)$ Kloosterman sums. This is somewhat reminiscent of the archimedean case where the (non-degenerate) ${\rm GL}(n)$-Whittaker function can be expressed as nested integrals of Bessel $K$-functions (i.e.\ ${\rm GL}(2)$ Whittaker functions), see \cite{St}. 
 
 The exact formula can also be used for certain exact evaluations (see below), it can be generalized to congruence subgroups, and can perhaps ultimately be a means for deeper methodological machinery in analytic number theory such as Poisson summation etc.  

\begin{cor}\label{cor3} With the notation as above, there exists a $\delta = \delta_n> 0$ such that 
$${\rm Kl}_{p}(  \psi, \psi', C^{\ast}w_l)  \ll \big(\max_{1 \leq j \leq n} |\psi_j|^{-1/2}_p  \big) \Big(\prod_{1 \leq k \leq n} p^{r_k}\Big)^{1-\delta}.$$
\end{cor}

Xinchen Miao kindly informed us that he obtained independently and by a slightly different method a similar bound in \cite[Theorem 5.1]{Mi}. 

In Subsection \ref{refined} we will give a quick argument that shows that one can choose $\delta \gg 1/n^2$. 
 It is an interesting question whether $\delta$ can be chosen independently of $n$. In this particular situation we do not know the answer, but for general Weyl elements the answer is NO, as we will see in the next subsection.



\subsection{Kloosterman sums for $G = {\rm GL}(n+1)$ for $w_{\ast}$}\label{sec13} We now establish similar results for the Weyl element $w_{\ast}$. For this Weyl element we have relatively severe restrictions on the moduli $C = (p^{r_1}, \ldots, p^{r_n})$, see \cite[Proposition 1.3]{Fr}. For instance, if  $\psi, \psi'$ satisfy  $p \nmid \psi_j\psi_j'$ and $n \geq 3$, then the exponents $r_k$ need to form an arithmetic progression. 

In the present situation we define
\begin{equation*}
\mc M_{w_{\ast}}(r) := \Big\{\ul m = (m_{ij})_{\substack{1\le i \le j\le n\\ i = 1 \text{ or } j = n}} \in \N_0^{2n-1} \mid \sum_{i\le k \le j} m_{ij} = r_k, \; 1\le k \le n\Big\}.
\end{equation*}
For $\ul m \in \mc M_{w_{\ast}}(r)$ we define
$$
\mc C_{w_{\ast}}(\ul m) = \cbm{\ul c = (c_{ij})_{\substack{ 1\le i \leq j \le n\\ i = 1 \text{ or } j = n}}}{\begin{array}{l} c_{1j} \in \Bbb{Z}/{\textstyle \prod_{k=j}^n} p^{m_{1k}}\Bbb{Z},\; 1\le j \le n,\\ c_{in}\in \Z/{\textstyle \prod_{k=2}^i} p^{m_{kn}}\Bbb{Z},\; 2\le i \le n,\end{array}  \; (c_{ij},p^{m_{ij}})=1}
$$
as well as the partial Kloosterman sum
${\rm Kl}_{p}(\underline{m}, \psi, \psi', w_{\ast})$ as 
\begin{equation}\label{partial1}
\begin{split}
\sum_{\underline{c} \in \mathcal{C}_{w_{\ast}}(\underline{m})} & e\Bigg(\psi_1 \frac{\overline{c_{11}}}{ p^{m_{11}}} + \sum_{j=2}^n   \psi_j \Big(\frac{c_{1(j-1)}\overline{c_{1j}}}{ p^{m_{1j}}} + \frac{c_{(j+1)n}\overline{c_{jn}}}{ p^{-m_{1(j-1)} + m_{1j} + m_{jn}}} \Big)\\
&+ 
 \psi'_1 \frac{c_{2n}}{ p^{m_{2n}}} + \psi'_n \Big( \frac{c_{1n}\overline{c_{nn}}}{ p^{m_{1n}}} + \frac{c_{1(n-1)}}{ p^{-m_{nn}+m_{1(n-1)}+m_{1n}}}\Big)\Bigg). 
 \end{split}
\end{equation}
Again we employ the convention that  $c_{ij} = 1$ and $m_{ij} = 0$ for $i>j$. Moreover, for $1 \leq i \leq j \leq n$ we  set $\overline{c_{ij}} = 0$ when $m_{ij}=0$. For the case $n=3$, we refer to \eqref{82}.

\begin{theorem}\label{thm2} For $C = (p^{r_1}, \ldots, p^{r_n})$ we have with the above notation
$${\rm Kl}_{p}(  \psi, \psi', C^{\ast}w_{\ast}) = \sum_{\underline{m} \in \mathcal{M}_{w_{\ast}}(r)} {\rm Kl}_{p}(\underline{m}, \psi, \psi', w_{\ast})$$
whenever the Kloosterman sum on the left hand side is well-defined.
\end{theorem}
 
 \begin{cor}\label{cor5} There exists a $\delta = \delta_n> 0$ such that 
$${\rm Kl}_{p}(  \psi, \psi', C^{\ast}w_{\ast})  \ll \big(\max_{1 \leq j \leq n}  |\psi_j|^{-1/2}_p \big)   \Big(\prod_{1 \leq k \leq n} p^{r_k}\Big)^{1-\delta}.$$
\end{cor}

Our proof will show that one can take $\delta \gg 1/n$. This should be seen in light of the following exact evaluation. 

 \begin{cor}\label{cor6} For $C = (p, \ldots, p)$ and characters $\psi, \psi'$ with $p \nmid \psi_j\psi_j'$ for $1 \leq j \leq n$ we have ${\rm Kl}_{p}(  \psi, \psi', C^{\ast}w_*)  = p^{n-1} + p^{n-2}$. 
 \end{cor}

This shows that Kloosterman sums can be fairly large. In particular, up to the constant the saving of size $\delta_n \asymp 1/n$ in Corollary \ref{cor5} is asymptotically  best possible. Corollary \ref{cor6} is a variation of \cite[Theorem 3]{Bl}, but proved by a completely different argument, since \cite[Theorem 3]{Bl} is special to the congruence subgroup $\Gamma_0(p)$. A similar, but more involved computation shows for instance
$${\rm Kl}_{p}(\psi, \psi', C^{\ast}w_l)  = (-1)^n (p+1)$$
for $C = (p, \ldots, p)$ and characters $\psi, \psi'$ with $p \nmid \psi_j\psi_j'$ for $1 \leq j \leq n$ (generalizing \cite[Property 4.10]{BFG} for $n=2$).

\subsection{An application}\label{appl} Kloosterman sums come up most prominently in the Kuznetsov formula. Up until now, in higher rank one could at best employ the trivial bound. With non-trivial bounds at hand, we can now refine the main result of \cite{AB}.   

The Ramanujan conjecture for the group ${\rm GL}(n)$ states that cuspidal automorphic representations  are tempered. This is way out of reach, but one may hope to prove that it holds in reasonable families   in the following quantitative average sense: automorphic forms that are ``far away'' from being tempered should occur ``rarely''. With this in mind, let $\mathcal{F}  =\mathcal{F}_{\Gamma(q)}(M)$ be the finite family of cusp forms $\varpi$ (eigenforms for the unramified Hecke algebra) for the principal congruence subgroup $\Gamma(q) \subseteq {\rm SL}_n(\Bbb{Z})$ with bounded spectral parameter $\| \mu_{\varpi, \infty}\|\leq M$ (see \cite{AB} for more details). For a place $v$ of $\Bbb{Q}$ and $\sigma \geq 0$ define
$$\mathcal{N}_v(\sigma, \mathcal{F}) = \#\{ \varpi \in \mathcal{F} \mid  \sigma_{\varpi,v} \geq \sigma\}$$
where
$\sigma_{\varpi,v} =  \max_j |\Re \mu_{\varpi,v}(j)|$. The trivial bound is $N_v(\sigma, \mathcal{F}) \leq N_v(0, \mathcal{F}) = \#\mathcal{F}$. On the other hand, we have $\sigma_{\text{triv}, v} =  \frac{1}{2}(n-1)$. One might hope that a trace formula can interpolate linearly between these two bounds: 
$$\mathcal{N}_v(\sigma, \mathcal{F}) \ll_{v, \varepsilon, n} M^{O(1)}  [{\rm SL}_n(\Bbb{Z}) : \Gamma(q)]^{1 - \frac{2\sigma}{n-1} + \varepsilon}$$
where $M^{O(1)}[{\rm SL}_n(\Bbb{Z}) : \Gamma(q)]$ should be thought of as a proxy for $\#\mathcal{F}$. 
This is a version of Sarnak's density conjecture \cite{Sar} and was proved in \cite{AB} for squarefree $q$. Of course, the trivial representation is not cuspidal, and so one may hope to do even better, but even for $n=2$ this is a very hard problem when the Selberg trace formula is employed. On the other hand, the Kuznetsov formula is better suited in this situation, since the residual spectrum is a priori excluded. We use this observation to go \emph{beyond} Sarnak's density conjecture in the scenario described above for prime moduli (the primality assumption is used before \eqref{new}).

\begin{prop}\label{prop1}  Let $M > 0$, $n \geq 5$.  Let $q$ be a large prime and let $\mathcal{F}=\mathcal{F}_{\Gamma(q)}(M)$ be the set of cuspidal automorphic forms for $\Gamma(q) \subseteq {\rm SL}_n(\Bbb{Z})$ with archimedean spectral parameter $\| \mu\| \leq M$. Fix a place $v$ of $\Bbb{Q}$. 
There exist constants $K, \delta$ depending only on $n$, such that 
$$\mathcal{N}_v(\sigma, \mathcal{F}) \ll_{v, n} M^K  [{\rm SL}_n(\Bbb{Z}) : \Gamma(q)]^{1 - \frac{(2+\delta)\sigma}{n-1}}$$
for $\sigma \geq 0$. 
\end{prop}

\section{The stratification of D{\k a}browski and Reeder}\label{section:DR}

Here we convert results in \cite{DR} to conform with our convention. Although nothing in this section is new, it might be convenient for the reader to compile some results. 

Let $G$ be a simply connected Chevalley group over $\Q_p$ with Lie algebra $\mf g$. Let $W$ be the Weyl group of $G$, and let $\Phi$ and $\Phi^+$ denote the set of roots and the set of positive roots of $\mf g$ respectively. We fix a set of simple roots $\Delta$ of $\mf g$. For each positive root $\beta \in \Phi^+$, there is a natural homomorphism $\phi_\beta: \SL(2) \to G$. For $t\in \Q_p$ we write
\begin{align*}
x_\beta(t) &= \phi_\beta\bp 1 & t\\ & 1\ep, & x_{-\beta}(t) &= \phi_\beta\bp 1\\ t & 1\ep.
\end{align*}
Through the canonical bijection $\beta \longleftrightarrow \check\beta$ between the set of roots $\Phi$ and the set of coroots $\check\Phi$, we have
$$\check\beta(c) = \phi_\beta\bp c \\ & c^{-1} \ep$$
for $c\in\Q_p^\times$. For $m\in\Z$, we write $m\check\beta$ for $\check\beta(p^m)$. This induces a natural embedding
$$ \check X := \Hom(\G_m, T) \hra T$$
from the set of cocharacters of $T$ into the maximal torus $T$. If $\lambda = -\sum_{\beta\in\Delta} r_\beta \check \beta \in \check X$ for $r_\beta \in \N_0$, we define the height of $\lambda$ to be
\begin{equation}\label{height}
 \height(\lambda) := \sum_{\beta\in\Delta} r_\beta.
 \end{equation}
For $w\in W$, we write 
$$R(w) := \cbm{\beta\in\Phi^+}{w\beta \in -\Phi^+}.$$
Then, if $w = s_{\beta_1}\ldots s_{\beta_l}$ is a reduced expression of $w$ as a product of simple reflections, then
\begin{equation}\label{eq:gdef}
R(w^{-1}) = \cbm{\gamma_j := s_{\beta_1}\ldots s_{\beta_{j-1}} \beta_j}{j = 1,\ldots, l}. 
\end{equation}

For $a\in\Q_p$ we set 
\begin{align*}
\mu(a) &:= \max\cb{0, -v_p(a)}, & a^* &:= \begin{cases} 0 & \text{ if } a \in \Z_p,\\ p^{-2v_p(a)} a^{-1} & \text{ if } a\not\in \Z_p.\end{cases}
\end{align*}
In particular, if $a \not \in \Z_p$, say $a = cp^{-m}$ for $m\ge 1$, $\varepsilon \in \Z_p^\times$, then we have $a^* = c^{-1} p^{-m}$. This means the map $a \mapsto a^*$ preserves $\mu$. For $\beta\in\Delta$ we define
\begin{equation}\label{eq:bbadef}
b_\beta(a) := x_\beta(a^*) (-\mu(a)\check\beta) \ol{s_\beta} x_\beta(a),
\end{equation}
where $\ol{s_\beta}$ is some fixed representative of the simple reflection $s_\beta$ in $N = N_G(T)$. When no confusion can arise, we simply write $s_\beta$ for $\ol{s_\beta}$. We observe that $b_\beta(a) \in G(\Z_p)$ for all $a\in\Q_p$. Indeed, we have
\begin{align*}
b_\beta(a) &= \phi_\beta\bp & -1\\ 1 & a\ep \text{ for } a \in \Z_p, & b_\beta(a) &= \phi_\beta\bp c^{-1} & \\ p^m & c \ep \text{ for } a = cp^{-m} \not\in \Z_p.
\end{align*}

Let $l \in \N$,  let $w = s_{\beta_1}\ldots s_{\beta_l}$ be a reduced representation of $w$ as a product of simple reflections, and write $\ul\beta = (\beta_1,\ldots, \beta_l)$. For $\ul a = (a_1,\ldots, a_l) \in \Q_p^l$, let $b(\ul a) := b_{\ul\beta}(\ul a)$ denote the image of $b_{\beta_1}(a_1) \ldots b_{\beta_l}(a_l)$ in $U(\Z_p)\bs G(\Q_p)$. By \cite[Proposition 2.1]{DR}, the map $b_{\ul\beta}:\Q_p^l \to U(\Z_p)\bs G(\Q_p)$ is injective, and its image is contained in $U(\Z_p) \bs \rb{BwU \cap G(\Z_p)}$, where $B = TU$ is the standard Borel subgroup of $G$. 

We remark that the map $b_{\ul \beta}$ really depends on the choice of the reduced representation of $w$. Moreover, $b_{\ul \beta}$ also depends on the choice of the representative $\ol{s_\beta}$ of $s_\beta$. If we fix a representative of $w$ in $G(\Z_p)$, then by convention we choose the representatives such that the toral part of the Bruhat decomposition of $b_{\ul \beta}(\ul a)$ has positive entries. 

For $\ul m = (m_1,\ldots, m_l) \in \N_0^l$, we set
$$Y_{\ul\beta} (\ul m) = \cbm{\ul a = (a_1,\ldots, a_l)\in\Q_p^l}{\mu(a_i) = m_i, \; i = 1,\ldots, l}.$$

Let $\beta\in\Delta$, and 
$$u = \prod_{\beta\in\Phi^+} x_\beta (a_\beta) \in U(\Q_p).$$
The \emph{$\beta$-coordinate function} $f_\beta$ is defined by $f_\beta(u) := a_\beta$. Note that this does not depend on the order of the product. For $a \in \Q_p$ and $u\in U(\Q_p)$, we define
$$R_\beta^a(u) := b_\beta(a) u b_\beta(a+f_\beta(u))^{-1}.$$
Then $R_\beta^a(U(\Z_p)) \sbe U(\Z_p)$. For a sequence $\ul\beta = (\beta_1,\ldots, \beta_l)$ of simple roots, we can construct a right action $* = *_{\ul\beta}: \Q_p^l \times U(\Z_p) \to \Q_p^l$, $\ul a \mapsto \ul a * u =: \ul a'$ as follows:
\begin{align*}
a'_l & a_l + f_{\beta_l} (u),\\
a'_j &= a_j + f_{\beta_j} \rb{R_{\beta_{j+1}}^{a_{j+1}}R_{\beta_{j+2}}^{a_{j+2}}\ldots R_{\beta_l}^{a_l}(u)}, & &1\le j \le l-1.
\end{align*}
Then $b_{\ul\beta}: \Q_p^l \to U(\Z_p)\bs G(\Q_p)$ is right $U(\Z_p)$-equivariant with respect to $*_{\ul\beta}$, that is, we have $b_{\ul\beta}(\ul a *_{\ul\beta} u) = b_{\ul\beta}(\ul a)u$ for $u\in U(\Z_p)$ and $\ul a \in \Q_p^l$. It follows that $b_{\ul\beta}$ induces a map
$$ \ul b_{\ul\beta}: \Q_p^l / U_w(\Z_p) \to U(\Z_p) \bs \rb{BwB \cap G(\Z_p)} / U_w(\Z_p),$$
where $\Q_p^l/U_w(\Z_p)$ denotes the set of $U_w(\Z_p)$-orbits with respect to the right action $*_{\ul\beta}$. 

For any cocharacter $\lambda \in \check X$, we define
\begin{equation*}
Y_\lambda = \coprod \Big\{Y_{\ul\beta} (\ul m) \mid  m \in \N_0^l, \lambda = -\sum_{j=1}^l m_j \check \gamma_j\Big\},
\end{equation*}
where $\gamma_j$ is defined in \eqref{eq:gdef}. Now we are able to state the main result in this section.
\begin{lemma}[{\cite[Proposition 3.3]{DR}}]
Let $\lambda\in\check X$, $n = \lambda w \in N(\Q_p)$, and let $w = s_{\beta_1}\ldots s_{\beta_l}$ be a reduced representation of $w$ as a product of simple reflections. Write $\ul\beta = (\beta_1,\ldots,\beta_l)$. Then $\ul b_{\ul\beta}$ gives a bijection between $Y_\lambda / U_w(\Z_p)$ and the Kloosterman set $X(n)$.
\end{lemma}
The following proposition gives the trivial bound for Kloosterman sums.
\begin{lemma}[{\cite[Proposition 3.4]{DR}}]\label{lem:repcount}
Assume the settings above. For $\ul m \in \N_0^l$ we have
$$\#\rb{Y_{\ul\beta}(\ul m)/U_w(\Z_p)} = p^{\height(\lambda)} \rb{1-p^{-1}}^{\kappa(\ul m)},$$
where $\kappa(\ul m)$ is the number of nonzero entries in $\ul m$ and $\height(\lambda)$ was defined in \eqref{height}. 
\end{lemma}

\section{Proof of Theorem \ref{thm1}}

We apply the results from Section \ref{section:DR} to our case. Let $G = \GL(n+1)$. A set of simple roots of $G$ is given by $\Delta = \cb{\alpha_1,\ldots,\alpha_n}$, where in usual notation $\alpha_i := e_i - e_{i+1}$. Using this root basis, the set of positive roots of $G$ is given by
$$\Phi^+ = \cbm{\alpha_{ij} := \alpha_i+\alpha_{i+1}+\ldots+\alpha_j}{1\le i \le j \le n}.$$
Throughout this section, we use the following reduced representation of $w_l$:
\begin{equation}\label{eq:wl_rrep}
w_l = (s_{\alpha_1}\ldots s_{\alpha_n})(s_{\alpha_1}\ldots s_{\alpha_{n-1}})\ldots (s_{\alpha_1}s_{\alpha_2}) s_{\alpha_1}.
\end{equation}
Recall the definition of $\gamma_j$ in \eqref{eq:gdef}. For the reduced representation \eqref{eq:wl_rrep} of $w_l$, we have
$$\ul\gamma = \rb{\alpha_{11},\alpha_{12},\ldots,\alpha_{1n},\alpha_{22},\alpha_{2n},\ldots,\alpha_{nn}}.$$

Now we give a characterisation for $b(\ul a)$, for $\ul a = (a_{11},\ldots, a_{1n}, a_{22},\ldots,a_{nn}) \in \Q_p^{ n(n+1)/2}$. Note that every $a_{ij}\in\Q_p$ can be written uniquely as   $a_{ij} = c_{ij} p^{-m_{ij}}$ with $m_{ij} \ge 0$, $c_{ij}\in\Z_p$, and $(c_{ij},p^{m_{ij}}) = 1$. 

\begin{lemma}\label{lem:wl_Bruhat}
Let $\ul a = (a_{11},\ldots,a_{nn}) \in \Q_p^{n(n+1)/2}$. Write $a_{ij} = c_{ij} p^{-m_{ij}}$, with $m_{ij} \ge 0$, $c_{ij}\in\Z_p$, and $(c_{ij},p^{m_{ij}}) = 1$. Then $b(\ul a)$ has a Bruhat decomposition $b(\ul a) = LNR$, where
$$L = \left(\begin{smallmatrix} 1 &L_{12} & L_{13} & \cdots & L_{1(n+1)}\\ & 1 & L_{23} & \cdots & L_{2(n+1)}\\   && \ddots&& \vdots\\ &&&1 & L_{n(n+1)}\\ &&&& 1\end{smallmatrix}\right), \quad N = \left(\begin{smallmatrix} && N_{1(n+1)}\\ &\Ddots\\ N_{(n+1)1} \end{smallmatrix}\right), \quad R = \left(\begin{smallmatrix} 1 &R_{12} & R_{13} & \cdots & R_{1(n+1)}\\ & 1 & R_{23} & \cdots & R_{2(n+1)}\\   && \ddots&& \vdots\\ &&&1 & R_{n(n+1)}\\ &&&& 1\end{smallmatrix}\right),$$
with
$$N_{i(n+2-i)} = (-1)^{(n+1-i)} \frac{\prod_{k=1}^{i-1} p^{m_{k(i-1)}}}{\prod_{k=1}^n p^{m_{ik}}}, \quad (1\le i \le n+1),$$
$$L_{ij} = \sum_{1\le \delta_1<\ldots < \delta_{j-i} \le j-1} L_{ij}\rb{\delta_1,\ldots,\delta_{j-i}}, \quad (1\le i < j \le n+1),$$
$$L_{ij}\rb{\delta_1,\ldots,\delta_{j-i}} = \frac{\prod_{k=1}^{j-i} c_{\delta_k(i-2+k)} \prod_{k=1}^{j-i} \prod_{t=\delta_{k-1}+1}^{\delta_k-1} p^{m_{t(i-2+k)}}}{\prod_{k=1}^{j-i} c_{\delta_k(i-1+k)} \prod_{k=1}^{\delta_{j-i}} p^{m_{k(j-1)}}},$$
$$R_{ij} = \sum_{n+1-i \le \partial_1 \le \ldots \le \partial_{j-i} \le n} R_{ij}\rb{\partial_1,\ldots,\partial_{j-i}}, \quad (1\le i < j \le n+1),$$
$$R_{ij} \rb{\partial_1,\ldots,\partial_{j-i}} = \frac{\prod_{k=1}^{j-i} c_{(n+2-i-k)\partial_k} \prod_{k=\partial_1+1}^n p^{m_{(n+2-i)k}}}{\prod_{k=1}^{j-i} c_{(n+3-i-k)\partial_k} \prod_{k=1}^{j-i} \prod_{t=\partial_k}^{\partial_{k+1}} p^{m_{(n+2-i-k)t}}}.$$
To interpret the formula above, we set $c_{ij}:= 1$ and $m_{ij}:=0$ if the condition $1\le i \le j \le n$ is not satisfied, and $\delta_0:= 0$, $\partial_{j-i+1} := n$. As a convention, when $m_{ij} = 0$ for $1 \leq i \leq j \leq n$, we define $c_{ij}^{-1} := 0$ as a formal symbol.
\end{lemma}

\textbf{Remarks:} The hard part is to find these explicit formulae for the Bruhat decomposition. Once the formulae are given, the proof is a straightforward inductive verification by simply matching terms on both sides of the matrix equation. The indices for $N$ look overly complicated, but are chosen in analogy with the ones in Lemma \ref{lem5} below. Our application of Lemma \ref{lem:wl_Bruhat} to the proof of Theorem \ref{thm1} will only require the values of $L_{ij}$ and $R_{ij}$ on the first off-diagonal, i.e.\ for $j-i = 1$ for which the formulae simplify substantially.  

\begin{proof}
First we justify the definition $c_{ij}^{-1} := 0$ as a formal symbol when $m_{ij} = 0$. We recall the definition of $b_\beta(a)$ for $\beta\in\Phi^+$, $a\in\Q_p$. We write $a$ as a product $a=cp^{-m}$, with $m\ge 0$, $c\in\Z_p$, and $(c,p^m) = 1$. When $m\ge 1$, we have
\begin{equation}\label{eq:bbcpm_expand}b_{\beta}(cp^{-m}) = \phi_\beta\rb{\bp 1 & c^{-1}p^{-m}\\ & 1\ep \bp & -p^m\\ p^m \ep \bp  1 & c p^{-m}\\ & 1\ep}.
\end{equation}
Meanwhile, when $m=0$ we have
$$b_{\beta}(c) = \phi_\beta\rb{\bp 1 & 0\\ & 1\ep \bp & -1\\ 1 \ep \bp  1 & c\\ & 1\ep}.$$
Hence \eqref{eq:bbcpm_expand} also works for $m=0$ if we treat $c^{-1} := 0$ as a formal symbol. As $b_{\ul\beta}(\alpha)$ is a product of such matrices, our convention is justified.

Now we prove the actual formula by induction. For easier manipulation, we assume $m_{ij}\ge 1$ for all $1\le i \le j \le n$; when some $m_{ij} = 0$ we use the convention $c_{ij}^{-1} := 0$ to the result. When $n=1$, the formula reads
$$b_{\alpha_1}(a_{11}) = \bp 1 & c_{11}^{-1} p^{-m_{11}}\\ & 1\ep \bp & p^{-m_{11}}\\ p^{m_{11}} \ep \bp 1 & c_{11} p^{-m_{11}}\\ & 1\ep,$$
which is precisely \eqref{eq:bbadef}. For the general case, let
$$\ul\beta_{[n]} := \rb{\alpha_1,\ldots,\alpha_n,\alpha_1,\ldots,\alpha_{n-1},\ldots,\alpha_1,\alpha_2,\alpha_1}$$
denote the reduced representation \eqref{eq:wl_rrep}. By induction, we have a Bruhat decomposition
$$b_{\ul\beta_{[n-1]}}(a_{22},\ldots,a_{nn}) = \bp L'\\ & 1\ep \bp N'\\ & 1\ep \bp R'\\ & 1\ep,$$
where the entries $L'$, $N'$, $R'$ are given by the formulae above, with indices $ij$ replaced with $(i+1)(j+1)$. By a slight abuse of notation, we shall denote the matrices above also by $L'$, $N'$, $R'$ respectively. On the other hand, it is straightforward to compute that
$$\Upsilon:= b_{(\alpha_1,\ldots,\alpha_n)}\rb{a_{11},\ldots,a_{1n}} = \rb{\begin{smallmatrix} \frac{1}{c_{11}}\\ p^{m_{11}} & \frac{c_{11}}{c_{12}}\\ & p^{m_{12}} & \frac{c_{12}}{c_{13}}\\ && \ddots & \ddots\\ &&& p^{m_{1n}} & c_{1n}\end{smallmatrix}}.$$

So it remains to show that
\begin{equation}\label{eq:LNR_ULNR}
LNR = \Upsilon L' N' R' = b_{\ul\beta_{[n]}}(\ul a)
\end{equation}
is indeed a Bruhat decomposition of $b_{\ul\beta_{[n]}}(\ul a)$. This is a straightforward brute force computation. For convenience, we provide the details. We expand
\begin{align*}
(LNR)_{ij} &= \sum_{k,\ell} L_{ik} N_{k\ell} R_{\ell j}, & (\Upsilon L' N' R')_{ij} &= \sum_{r,k,\ell} \Upsilon_{ir} L'_{rk} N'_{k\ell} R'_{\ell j}. 
\end{align*}
As each row of $N$ has exactly one nonzero entry, we can collapse the sum and write
\begin{equation}\label{eq:LNR_sum}
(LNR)_{ij} = \sum_k L_{ik} N_{k(n+2-k)} R_{(n+2-k)j} = \sum_k \sum_{\delta,\partial} L_{ik}(\delta) R_{(n+2-k)j}(\partial).
\end{equation}
By the same argument, we write
\begin{multline}\label{eq:ULNR_sum}
(\Upsilon L'N'R')_{ij} = \frac{c_{1(i-1)}}{c_{1i}} \rb{\sum_{k=1}^n L'_{ik} N'_{k(n+1-k)} R'_{(n+1-k)j} + L'_{i(n+1)} R'_{(n+1)j}}\\
+ p^{m_{1(i-1)}} \rb{\sum_{k=1}^n L'_{(i-1)k} N'_{k(n+1-k)} R'_{(n+1-k)j} + L'_{(i-1)(n+1)} R'_{(n+1)j}}.
\end{multline}

Since
$$L'_{i(n+1)} R'_{(n+1)j} = \begin{cases} 1 & \text{ if } i=j=n+1,\\ 0 & \text{ otherwise,}\end{cases}$$
it follows that for $1\le i,j \le n$ we have
\ba
(\Upsilon L'N'R')_{ij} = &\;\frac{c_{1(i-1)}}{c_{1i}} \sum_{k=1}^n L'_{ik} N'_{k(n+1-k)} R'_{(n+1-k)j} + p^{m_{1(n-1)}}\sum_{k=1}^n L'_{(i-1)k} N'_{k(n+1-k)} R'_{(n+1-k)j}\\
= &\;\frac{c_{1(i-1)}}{c_{1i}} \sum_{k=1}^n \sum_{\delta', \partial'} L'_{ik}(\delta') N'_{k(n+1-k)} R'_{(n+1-k)j}(\partial')\\
&\hspace{3cm} + p^{m_{1(n-1)}}\sum_{k=1}^n \sum_{\delta'', \partial''} L'_{(i-1)k}(\delta'') N'_{k(n+1-k)} R'_{(n+1-k)j}(\partial''),
\ea
where
\begin{align*}
L'_{ij}(\delta) &= \frac{\prod_{k=1}^{j-i} c_{(\delta_k+1)(i-1+k)} \prod_{k=1}^{j-i} \prod_{t=\delta_{k-1}+1}^{\delta_k-1} p^{m_{(t+1)(i-1+k)}}}{\prod_{k=1}^{j-i} c_{(\delta_k+1)(i+k)} \prod_{k=1}^{\delta_{j-i}} p^{m_{(k+1)j}}}, & R'_{ij}(\partial) &= R_{ij}(\partial).
\end{align*}
It is then straightforward to verify that
\begin{multline*}
L_{ik}(1,\delta_2,\ldots,\delta_{k-i}) N_{k(n+2-k)} R_{(n+2-k)j}(\partial)\\
= \frac{c_{1(i-1)}}{c_{1i}} L'_{i(k-1)}(\delta_2-1,\ldots,\delta_{k-i}-1) N_{(k-1)(n+2-k)} R'_{(n+2-k)j}(\partial),
\end{multline*}
and
\begin{multline*}
L_{ik}(\delta_1,\ldots,\delta_{k-i}) N_{k(n+2-k)} R_{(n+2-k)j}(\partial) \\
= p^{m_{1(i-1)}} L'_{(i-1)(k-1)}(\delta_1-1,\ldots,\delta_{k-i}-1) N_{(k-1)(n+2-k)} R'_{(n+2-k)j}(\partial)
\end{multline*}
if $\delta_1\ge 2$. Matching the terms with \eqref{eq:LNR_sum} yields $(LNR)_{ij} = (\Upsilon L'N'R')_{ij}$ for $1\le i,j\le n$.

Now consider the case where $1\le i \le n$, $j=n+1$. From \eqref{eq:ULNR_sum}, we deduce that
\ba
(\Upsilon L' N' R')_{i(n+1)} = 0, \quad 1\le i \le n.
\ea
It remains to show that $(LNR)_{i(n+1)} = 0$ for $1\le i \le n$. By straightforward computation, we have
\begin{multline*}
L_{ik}(\delta_1,\ldots,\delta_{k-i}) N_{k(n+2-k)} R_{(n+2-k)j}(\partial_1,\ldots,\partial_{k-1})\\
= -L_{ik}(\delta_1,\ldots,\delta_{k-i+1}) N_{(k+1)(n+1-k)} R_{(n+1-k)j}(\partial'_1,\ldots,\partial'_k),
\end{multline*}
where $\partial'_1 = k$, $\partial'_\ell = \max\cb{k,\partial_{\ell-1}}$ for $2\le \ell \le k$, and $\delta_{k-i+1} = k + \sum_{\ell=1}^{k-1} \partial_\ell - \sum_{\ell=2}^k \partial'_\ell$. Putting this back into \eqref{eq:LNR_sum} yields $(LNR)_{i(n+1)} = 0$ as desired.

Now consider the case where $i=n+1$, $1\le j \le n$. Then \eqref{eq:LNR_sum} and \eqref{eq:ULNR_sum} say
\ba
(\Upsilon L' N' R')_{(n+1)j} &= p^{m_{1n}} \sum_{k=1}^n L'_{nk} N'_{k(n+1-k)} R'_{(n+1-k)j} = p^{m_{1(i-1)}} N'_{n1} R'_{1j}.\\
(LNR)_{(n+1)j} &= \sum_{k=1}^{n+1} L_{(n+1)k} N_{k(n+2-k)} R_{(n+2-k)j} = N_{(n+1)1} R_{1j}. 
\ea
Since $r_{1j}(\partial) = r'_{1j}(\partial)$, and $N_{(n+1)1} = p^{m_{1n}} N'_{n1}$, it follows that $(\Upsilon L' N' R')_{(n+1)j} = (LNR)_{(n+1)j}$.

Finally, for $i=j=n+1$, we have
\ba
(LNR)_{(n+1)(n+1)} = N_{(n+1)1} R_{1(n+1)} = c_{1n} = (\Upsilon L'N'R')_{(n+1)(n+1)}.
\ea
So \eqref{eq:LNR_ULNR} holds, finishing the proof.
\end{proof}

\begin{lemma}\label{lem:wl_coset}
Assume the settings above. For $\ul m \in \N_0^{n(n+1)/2}$, a complete system of coset representatives for $Y_{\ul\beta}(\ul m)/U(\Z_p)$ is given by
\[
\cbm{(c_{ij}p^{-m_{ij}})_{1\le i\le j\le n}}{
c_{ij}\,  \Big(\text{{\rm mod }} \prod_{k=j}^n p^{m_{ik}}\Big), 
 (c_{ij},p^{m_{ij}})=1}.
\]
\end{lemma}

\textbf{Remark:} The shape of the system is not completely obvious (to us), but once it is given, the verification is somewhat lengthly, but straightforward.

\begin{proof}
From Lemma \ref{lem:repcount} we already know the number of coset representatives needed. So it remains to show that all these coset representatives are inequivalent under right action by $U(\Z_p)$. Again we argue inductively. The case $n=1$ is straightforward to verify. Indeed, suppose we have
\[
\bp 1 & c'_{11}p^{-m_{11}}\\ & 1\ep = \bp 1 & c_{11}p^{-m_{11}}\\ & 1\ep \bp 1 & u_{11}\\ & 1\ep
\]
for some $u_{11}\in \Z_p$. This actually says
\[
c'_{11}p^{m_{11}} = c_{11} p^{m_{11}} + u_{11},
\]
which implies $u_{11}=0$, and $c'_{11} = c_{11}$ as desired.

Now we consider the general case. For $n=r$, we set
\[
R := b_{\ul\beta}(c_{11}p^{-m_{11}},\ldots,c_{rr}p^{-m_{rr}}), \quad R' := b_{\ul\beta}(c'_{11}p^{-m_{11}},\ldots,c'_{rr}p^{-m_{rr}}),
\]
and 
\begin{equation*}
u = \rb{\begin{smallmatrix} 1 & u_{11} & \cdots & u_{1r}\\ & 1 & \cdots & u_{2r}\\ && \ddots & \vdots\\ &&& 1\end{smallmatrix}},
\end{equation*}
such that 
\begin{equation}\label{eq:cpcu}
R' = Ru.
\end{equation}
Removing the final column and the final row of the matrices yields the problem for $n=r-1$ (with a renaming of variables $c_{ij}\mapsto c_{(i-1)(j-1)}, c'_{ij}\mapsto c'_{(i-1)(j-1)}$). By induction, we deduce that the first $r$ rows and columns of $R$ and $R'$ are identical, and $u_{ij} = 0$ for all $1\leq i \le j\leq r-1$. Using Lemma \ref{lem:wl_Bruhat}, we deduce that $c_{ij} = c'_{ij}$ for $2\leq i\leq j\leq r$. 

It remains to consider the final columns of the matrices. The $(1,r+1)$-th entry of \eqref{eq:cpcu} reads
\begin{equation}\label{eq:cpcu1}
c'_{1r} \prod_{j=1}^r p^{-m_{jr}} = c_{1r} \prod_{j=1}^r p^{-m_{jr}} + \sum_{j=2}^{r+1} u_{(r+2-j)r} c_{jr} \prod_{k=j}^r p^{-m_{kr}},
\end{equation}
where again we set $c_{(r+1)r} := 1$. From \eqref{eq:cpcu1}, we deduce that $c'_{1r} = c_{1r}$. It then follows that
\begin{equation}\label{eq:urrSolve}
u_{rr} = -R_{1r}^{-1} \rb{u_{(r-1)r} R_{1(r-1)} + \ldots + u_{2r}R_{12} + u_{1r}}= -\sum_{j=1}^{r-1} u_{jr} \frac{c_{(r+2-j)r}}{c_{2r}} \prod_{k=2}^{r+1-j} p^{m_{kr}}. 
\end{equation}

Now we turn to the $(2,r+1)$-th entry of \eqref{eq:cpcu}. It reads
\begin{equation*}
R'_{2(r+1)} = R_{2(r+1)} + \sum_{j=2}^r u_{jn} R_{2j}.
\end{equation*}
From Lemma \ref{lem:wl_Bruhat}, we see that there is exactly one term in $R_{2(r+1)}$ and $R'_{2(r+1)}$ that depends on $c_{1j}$ for some $j\le r-1$. Since $c_{ij} = c'_{ij}$ for $2\le i\le j \le r$, we can remove all the other terms within $R_{2(r+1)}$ and $R'_{2(r+1)}$, and obtain
\begin{equation*}
\frac{c'_{1(r-1)} p^{m_{rr}-m_{1(r-1)}-m_{1r}}}{\prod_{k=2}^{r-1} p^{m_{k(r-1)}}} = \frac{c_{1(r-1)} p^{m_{rr}-m_{1(r-1)}-m_{1r}}}{\prod_{k=2}^{r-1} p^{m_{k(r-1)}}}+ \sum_{j=2}^r u_{jr} R_{2j}.
\end{equation*}
To show that $c'_{1(r-1)} = c_{1(r-1)}$, it suffices to prove that
\[
\sum_{j=2}^r u_{jr} R_{2j} \in p^{m_{rr}} \prod_{k=2}^{r-1} p^{-m_{k(r-1)}} \Z_p.
\]
Using \eqref{eq:urrSolve}, we rewrite
\begin{equation}\label{eq:cpcu2val}
\sum_{j=2}^r u_{jr} R_{2j} = \sum_{j=1}^{r-1} u_{jr} \rb{R_{2j} - \frac{R_{1j} R_{2r}}{R_{1r}}}.
\end{equation}
We expand
\[
R_{2j} = \sum_{n-1\le \partial_1 \le \ldots \le \partial_{j-2} \le n} R_{2j}(\partial),
\]
where
\[
R_{2j}(\partial) =\frac{c_{(r-1)\partial_1} \ldots c_{(r+2-j)\partial_{j-2}}}{c_{r\partial_1}\ldots c_{(r+3-j)\partial_{j-2}}} \frac{1}{p^M} \begin{cases} p^{m_{rr}}, &\partial_1 \not= r,\\ 1, & \partial_1 = r.\end{cases} 
\]
where 
$$M = m_{(r-1)\partial_1} + \ldots + m_{(r-1)\partial_2} + m_{(r-2)\partial_2} + \ldots + m_{(r-2)\partial_3} + \ldots +m_{(r+2-j) \partial_{j-2}} + \ldots + m_{(r+2-j)r}.$$
For $0\le l \le j-2$, we write $\partial_{[j]}(l) = (\partial_1,\ldots,\partial_{j-2})$, with $\partial_k = r-1$ if $k\le l$, and $\partial_k = r$ otherwise. From \eqref{eq:urrSolve}, it is easy to check that 
\[
\frac{R_{1j}R_{2r}(\partial_{[r]}(k))}{R_{1r}} \in p^{m_{rr}} \prod_{k=2}^{r-1} p^{-m_{k(r-1)}} \Z_p.
\]
for $k\ge j-1$. On the other hand, we verify that
\[
R_{2j}(\partial_{[j]}(k)) = \frac{R_{1j}R_{2r}(\partial_{[r]}(k))}{R_{1r}}
\]
for $0\le j \le k-2$. So we conclude that
\[
R_{2j} - \frac{R_{1j}R_{2r}}{R_{1r}} \in p^{m_{rr}} \prod_{k=2}^{r-1} p^{-m_{k(r-1)}} \Z_p
\]
for $1\le j \le r-1$. The claim then follows from \eqref{eq:cpcu2val}.

By similar arguments, we proceed inductively and show that
\[
\sum_{j=i}^r u_{jr} R_{ij} \in \prod_{k=r+2-i}^n p^{m_{(r+2-i)k}} \prod_{k+2}^{r+1-i} p^{-m_{k(r+1-i)}} \Z_p
\]
for $3\le i \le r$, and thus $c'_{1j} = c_{1j}$ for $1\le j \le r-2$. This finishes the proof of the statement.
\end{proof}

Combining the previous computations, we complete the proof of Theorem \ref{thm1}, noting  that $\lambda  = -\sum_{j=1}^n r_j \check{\alpha}_j \in \check{X}$, $r_j \geq 0$ corresponds to the components of $r = (r_1, \ldots, r_n)$.  From Lemma \ref{lem:wl_coset} we obtain the summation condition in \eqref{partial} and from Lemma \ref{lem:wl_Bruhat} the shape of the exponential for two characters as in \eqref{char}.

\section{Proof of Theorem \ref{thm2}}

The proof of Theorem \ref{thm2} is similar. We omit the analogous straightforward verification and just write down the relevant formulae. We fix a reduced representation of $w_{\ast}$ as follows
\begin{equation*}
w_{\ast} = s_{\alpha_1} s_{\alpha_2} \ldots s_{\alpha_{n-1}} s_{\alpha_n} s_{\alpha_{n-1}} \ldots s_{\alpha_2} s_{\alpha_1}.
\end{equation*}
Recall the definition of $\gamma_j$ in \eqref{eq:gdef}. For the reduced representation  of $w_*$, we have
$$\ul\gamma = \rb{\alpha_{11},\alpha_{12},\ldots,\alpha_{1n},\alpha_{nn},\ldots,\alpha_{2n}}.$$

Now we give a characterisation for $b(\ul a)$, for $\ul a = (a_{11},\ldots, a_{1n}, a_{nn},\ldots,a_{2n}) \in \Q_p^{2n-1}$. Again, every $a_{ij}\in\Q_p$ can be written uniquely as  $a_{ij} = c_{ij} p^{-m_{ij}}$, with $m_{ij} \ge 0$, $c_{ij}\in\Z_p$, and $(c_{ij},p^{m_{ij}}) = 1$. 

\begin{lemma}\label{lem5}
Let $\ul a = (a_{11},\ldots,a_{1n},a_{nn},\ldots,a_{2n}) \in \Q_p^{2n-1}$. Write $a_{ij} = c_{ij} p^{-m_{ij}}$, with $m_{ij} \ge 0$, $c_{ij}\in\Z_p$, and $(c_{ij},p^{m_{ij}}) = 1$. Then $b(\ul a)$ has a Bruhat decomposition $b(\ul a) = LNR$, where
$$L = \left(\begin{smallmatrix} 1 &L_{12} & L_{13} & \cdots & L_{1(n+1)}\\ & 1 & L_{23} & \ldots & L_{2(n+1)}\\   && \ddots&& \vdots\\ &&&1 & L_{n(n+1)}\\ &&&& 1\end{smallmatrix}\right), \quad N = \left(\begin{smallmatrix} &&&& N_{1(n+1)}\\ &N_{22}\\ &&\ddots\\ &&&N_{nn}\\ N_{(n+1)1}\end{smallmatrix}\right), \quad R = \left(\begin{smallmatrix} 1 &R_{12} & R_{13} & \cdots & R_{1(n+1)}\\ & 1 & R_{23} & \ldots & R_{2(n+1)}\\   && \ddots&& \vdots\\ &&&1 & R_{n(n+1)}\\ &&&& 1\end{smallmatrix}\right),$$
where
\begin{align*}
N_{1(n+1)} &= (-1)^n \prod_{j=1}^n p^{-m_{1j}}, & N_{ii} &= - p^{m_{1(i-1)}-m_{in}} \quad (2\le i \le n), & N_{(n+1)1} &= \prod_{j=1}^n p^{m_{jn}},
\end{align*}
\begin{align*}
L_{1j} &= (-1)^j c_{1(j-1)}^{-1} \prod_{k=1}^{j-1} p^{-m_{1k}} \quad (2\le j\le n),  & L_{1(n+1)} &= c_{11}^{-1} c_{2n}^{-1} \prod_{k=1}^n p^{-m_{kn}},
\end{align*}
$$L_{ij}= (-1)^{j-i+1} \rb{\frac{c_{1(i-1)}}{c_{1(j-1)}} \prod_{k=i}^{j-1} p^{-m_{1k}} + \frac{c_{1i} c_{(i+1)n}}{c_{1(j-1)} c_{in}} p^{m_{1(i-1)}-m_{in}} \prod_{k=i}^{j-1} p^{-m_{1k}}} \quad ( 2\le i < j \le n),$$
$$L_{i(n+1)} = \frac{c_{1(i-1)}}{c_{1i}c_{(i+1)n}} p^{-m_{1n}} \prod_{k=i+1}^n p^{-m_{kn}} + c_{in}^{-1} p^{m_{1(i-1)}-m_{1n}} \prod_{k=i}^n p^{-m_{in}} \quad  (2\le i \le n),$$
\begin{align*}
R_{1j} &= c_{jn} \prod_{k=2}^j p^{-m_{kn}}\quad (2\le j \le n), & R_{1(n+1)} &= c_{1n} \prod_{k=1}^n p^{-m_{kn}},
\end{align*}
$$R_{i(n+1)} = (-1)^{n-i} \rb{\frac{c_{1i} c_{(i+1)n}}{c_{in}} \prod_{k=i}^n p^{-m_{1k}} + c_{1(i-1)} p^{m_{in}} \prod_{k=i-1}^n p^{-m_{1k}}} \quad (2\le i \le n).$$
To interpret the formula above, we set $c_{ij}:= 1$ and $m_{ij}:=0$ if the condition $1\le i \le j \le n$ is not satisfied. As a convention, when $m_{ij} = 0$ for $1 \leq i \leq j \leq n$, we define $c_{ij}^{-1} := 0$ as a formal symbol.
\end{lemma}
\begin{proof}
Similar as the proof of Lemma \ref{lem:wl_Bruhat}.
\end{proof}

\begin{lemma}
Assume the settings above. For $\ul m \in \N_0^{2n-1}$, a complete system of coset representatives for $Y_{\ul\beta}(\ul m)/U(\Z_p)$ is given by
\ba
\cbm{(c_{ij}p^{-m_{ij}})_{\substack{ i=1, 1\le j \le n\\ 2\le i\le n, j=n}}}{\begin{array}{l} c_{1j} \rb{\mod\textstyle \prod_{k=j}^n p^{m_{1k}}},\; 1\le j \le n,\\ c_{in} \rb{\mod\textstyle \prod_{k=2}^i p^{m_{kn}}},\; 2\le i \le n,\end{array},\; (c_{ij},p^{m_{ij}})=1}.
\ea
\end{lemma}
\begin{proof}
Similar as the proof of Lemma \ref{lem:wl_coset}.
\end{proof}

Combining the above results, we complete the proof of Theorem \ref{thm2}. 

\section{Non-trivial bounds for Kloosterman sums}\label{non-trivial}

\subsection{General preparation}
 In this section we prove Corollaries \ref{cor3} and \ref{cor5}. We first prove Corollary \ref{cor3}. The idea is that the partial Kloosterman sum ${\rm Kl}_{p}(\underline{m}, \psi, \psi', w_l)$ defined in \eqref{partial} with
\begin{equation}\label{rm}
\sum_{i \leq k \leq j} m_{ij} = r_k, \quad 1 \leq k \leq n, 
\end{equation}
is a nested sum of classical ${\rm GL}(2)$ Kloosterman sums, for which we have Weil's bound \eqref{weil} available.  We start with a simple lemma.

\begin{lemma}\label{lem7}  Let $\gamma_1, \gamma_2 \geq 0$, $b_1, b_2 \in \Bbb{Z}$ with $\min(b_1, b_2) < 0$. Then
$$\sum_{c_1 = 1}^{p^{\gamma_1}} \sum_{c_2 = 1}^{p^{\gamma_2}} |c_1 p^{b_1} + c_2 p^{b_2}|^{-1/2}_p \ll  p^{\gamma_1+\gamma_2 + \frac{1}{2}\min(b_1, b_2)}.$$
\end{lemma}

\begin{proof} The sum on the left hand side equals
\begin{equation}\label{sum}
\sum_{\delta_1 = 0}^{\gamma_1} \sum_{\delta_2 = 0}^{\gamma_2} \sum_{\substack{c_1 = 1\\ (c_1, p) = 1}}^{p^{\gamma_1 - \delta_1}} \sum_{\substack{c_2 = 1\\ (c_2, p) = 1}}^{p^{\gamma_2 - \delta_2}} |c_1 p^{b_1 + \delta_1} + c_2 p^{b_2 + \delta_2}|^{-1/2}_p .
\end{equation}
Suppose for notational simplicity that $b_2 \leq b_1$ (the other case is completely analogous). 
 Let us first assume that $b_1 + \delta_1 \not= b_2 + \delta_2$. Then the two inner sums are bounded by$$p^{\gamma_1 - \delta_1 + \gamma_2 - \delta_2 + \frac{1}{2}\min(\delta_1 + b_1, \delta_2 + b_2)} \leq p^{\gamma_1  + \gamma_2  + \frac{1}{2} b_2 -  \delta_1 - \frac{1}{2} \delta_2}$$
where the second inequality can be seen by distinguishing the cases 
 $\delta_2 + b_2 \leq \delta_1 + b_1$ and $\delta_2 + b_2 > \delta_1 + b_1$. 

Let us now assume $b_1 + \delta_1 = b_2 + \delta_2$. Then the inner two sums are  at most
\begin{displaymath}
\begin{split}
& p^{\frac{1}{2}(b_2 + \delta_2)} \sum_{\delta \leq \max(\gamma_2 - \delta_2, \gamma_1 - \delta_1)} p^{\delta/2} \underset{p^{\delta} \mid c_1 + c_2}{\sum_{ c_1 = 1 }^{p^{\gamma_1 - \delta_1}} \sum_{ c_2 = 1 }^{p^{\gamma_2 - \delta_2}}} 1\\
 \leq & p^{\frac{1}{2}(b_2 + \delta_2)} \sum_{\delta \leq \max(\gamma_2 - \delta_2, \gamma_1 - \delta_1)} p^{\delta/2}  \Big(p^{\min(\gamma_1 - \delta_1, \gamma_2 - \delta_2)} + \frac{p^{\gamma_1 - \delta_1 + \gamma_2 - \delta_2}}{p^{\delta}}\Big) \\
\ll & p^{\frac{1}{2}(b_2 + \delta_2)}\big(p^{\frac{1}{2}(\gamma_1 + \gamma_2 - \delta_1 - 2\delta_2) } + p^{\gamma_1 + \gamma_2 - \delta_1 - \delta_2}\big). 
\end{split}
\end{displaymath}
(For $p = 2$ the $\delta$-sum runs up to   $\max(\gamma_2 - \delta_2, \gamma_1 - \delta_1) + 1$.) Thus in all cases we bound \eqref{sum} by
$$\ll \sum_{\delta_1 = 0}^{\gamma_1} \sum_{\delta_2 = 0}^{\gamma_2} p^{\gamma_1 + \gamma_2 + \frac{1}{2}\min(b_1, b_2) - \frac{1}{2}(\delta_1 + \delta_2)} \ll p^{\gamma_1 + \gamma_2 + \frac{1}{2}\min(b_1, b_2)},$$
and the lemma follows. \end{proof}

We return to the partial Kloosterman sum \eqref{partial} for the long Weyl element. Let 
\begin{equation}\label{cij}
C_{ij} = p^{m_{ij} + \ldots + m_{in}}
\end{equation}
be the modulus of the $c_{ij}$-sum, for any $1 \leq i \leq j \leq n$. For $j < i$ we put $C_{ij} = 1$. 

Let us fix one variable $c_{ij}$. Then the $c_{ij}$-sum in \eqref{partial} is given by
\begin{equation}\label{sigmaij}
\Sigma_{ij} := \sum_{\substack{1 \leq c_{ij} \leq C_{ij}\\ (c_{ij}, p^{m_{ij}}) = 1}} e(c_{ij} A + \bar{c}_{ij}B)
\end{equation}
where 
\begin{displaymath}
\begin{split}
A  = A_{ij}&=  \psi_{j+1} \bar{c}_{i, j+1} p^{a_1} + \psi'_{n+1-i}\bar{c}_{i+1, j} p^{a_2},  \\
B = B_{ij}  &=  \psi_j  c_{i, j-1} p^{b_1} +\psi'_{n+2-i} c_{i-1, j} p^{b_2}  
\end{split}
\end{displaymath}
with
\begin{equation}\label{a1etc}
\begin{split}
a_1 = a_1(i, j) = & m_{1j} + \ldots +m_{i-1, j} - m_{1, j+1} - \ldots - m_{i, j+1},\\
a_2 =  a_2(i, j) = & m_{i+1, n} + \ldots +m_{i+1, j+1} - m_{i, n} - \ldots - m_{i, j},\\
b_1 = b_1(i, j) = & m_{1,j-1} + \ldots +m_{i-1, j-1} - m_{1, j} - \ldots - m_{i, j},\\
b_2 = b_2(i, j) = & m_{i, n} + \ldots +m_{i, j+1} - m_{i-1, n} - \ldots - m_{i-1, j}.
\end{split}
\end{equation}
Here we apply the following conventions, in this order: if $m_{ij} = 0$ for some $1 \leq   i \leq j \leq n$, we put $\overline{c_{ij}} = 0$. If $j < i$, we put $m_{ij} = 0$ and $c_{ij} = 1$ and $C_{ij} = 1$. If none of the above cases apply, and $i < 0$ or $j > n$, we put $c_{ij} = \overline{c_{ij}}= m_{ij} = 0$ and $C_{ij} = 1$. 

Let us   assume $$m_{ij} \not= 0.$$  Let $v_p(A) = -\alpha$, $v_p(B) = -\beta$. Assume without loss of generality $\alpha \geq \beta$, the other case is analogous. If $\alpha \leq 0$, then trivially $|\Sigma_{ij}| \leq C_{ij}$. If $\alpha > 0$, we extend the range of summation to avoid issues of well-definedness, and obtain by Weil's bound
\begin{displaymath}
\begin{split}
|\Sigma_{ij}| &= \Big|p^{-\alpha}\sum_{\substack{1 \leq c_{ij} \leq C_{ij}p^{\alpha}\\ (c_{ij}, p)= 1}} e\Big(\frac{c_{ij} A C_{ij}p^{\alpha} + \bar{c}_{ij}B C_{ij} p^{\alpha}}{C_{ij} p^{\alpha}}\Big)\Big| \\
&\leq 2 p^{-\alpha} (C_{ij}p^{\alpha})^{1/2} (C_{ij}, C_{ij}p^{\alpha-\beta}, C_{ij}p^{\alpha})^{1/2} = 2 C_{ij} p^{-\alpha/2}.
\end{split}
\end{displaymath}
We  conclude in all cases (still assuming $m_{ij} \not = 0$)
\begin{equation}\label{kloosterbound}
|\Sigma_{ij}| \leq 2C_{ij} \min(1, |A_{ij}|^{-1/2}_p, |B_{ij}|^{-1/2}_p).
\end{equation}
Note that this uses no specific information about $A$ and $B$ and holds for any sum of the type \eqref{sigmaij}.\\

From this and the previous lemma we see  that 
\begin{equation*}
\begin{split}
\sum_{c_{i, j-1}} \sum_{c_{i-1, j}} |\Sigma_{ij}| 
&\leq\big(\max_{1 \leq j \leq n}|\psi_j|_p^{-1/2}\big)  C_{i, j-1} C_{i-1, j} C_{ij}  p^{\frac{1}{2}\min(0, b_1(i, j))}
\end{split}
\end{equation*}
if  $i \not= 1$ (in which case $i-1 > 0$). Note that this continues to hold for $i = j$
  by our general conventions. If $i = 1$ (in which case $c_{i-1, j} = 0$), a similar, but simpler argument confirms the bound, too. Here we dropped potential savings in the exponents $a_1, a_2, b_2$.

If $m_{ij} = 0$ we simply estimate trivially, and therefore obtain in all cases the bound
\begin{equation}\label{klooster}
\begin{split}
\sum_{c_{i, j-1}} \sum_{c_{i-1, j}} |\Sigma_{ij}| 
&\leq \big(\max_{1 \leq j \leq n}|\psi_j|_p^{-1/2}\big)  C_{i, j-1} C_{i-1, j} C_{ij}  p^{\frac{1}{2}b^{\ast}_1(i, j)}
\end{split}
\end{equation}
where
 $$b_*(i, j) := \delta_{m_{ij} \not= 0}\min(0, b_1(i,j)).$$


\subsection{A soft argument}\label{sec52} With a view towards possible generalizations we first demonstrate a  soft argument. We define an ordering on the set of indices $(i, j)$, $1 \leq i \leq j \leq n$ as follows
\begin{equation}\label{order}
(1,  n) < (1, n-1) < \ldots < (1, 1) < (2, n) < \ldots < (2, 2) < \ldots < (n, n).
\end{equation}
Let $$\mu_{ij} = \max_{(\alpha, \beta) < (i, j)} m_{\alpha\beta}.$$
Then \eqref{klooster} implies 
 \begin{equation*}
\sum_{\substack{c_{\nu\mu}\\ (\nu, \mu) \not= (i, j)}} \Big| \sum_{c_{ij}} (...) \Big| \ll\big(\max_{1 \leq j \leq n} |\psi_j|^{-1/2}_p  \big) \Big(\prod_{1 \leq i \leq j \leq n} C_{ij}\Big) p^{-m_{ij}/2+ O(\mu_{ij})}
\end{equation*}
(which is trivially true if $m_{ij} = 0$) for any $(i, j)$. 
Choosing the index pair $(i, j)$ suitably, we conclude
$${\rm Kl}_p(\underline{m}, \psi, \psi', w_l) \ll \big(\max_{1 \leq j \leq n} |\psi_j|^{-1/2}_p  \big) p^{r_1 + \ldots + r_n -\delta_0 \max_{i, j} m_{ij}}$$
for some $\delta_0 > 0$ (depending on $n$), from which we easily obtain the statement  of  Corollary \ref{cor3}, observing that the number of $\underline{m}$ for a given vector $r$ is $O(p^{\varepsilon(r_1 + \ldots + r_n)})$ for every $\varepsilon > 0$. 

\subsection{A  refined argument}\label{refined}  The previous argument uses cancellation only in one index pair $(ij)$. It is very flexible and requires only the ordering \eqref{order}, but no further computations. On the other hand, it gives only a small value of $\delta$ (exponentially decreasing in $n$). A more refined argument runs as follows. 
We partition the index pairs into 4 classes depending on the parity of $i$ and $j$ and obtain
  $${\rm Kl}_{p}(\underline{m}, \psi, \psi', w_l) \ll \Big(\prod_{1 \leq i \leq j \leq n} C_{ij}\Big)  \Big(\prod_{\substack{1 \leq i \leq j \leq n\\ i \equiv i_0 \, (\text{mod } 2)\\ j \equiv j_0 \, (\text{mod } 2)}} p^{\frac{1}{2}b_{\ast}(i, j)}\Big) $$
  for $i_0, j_0 \in \{0, 1\}$. Recall that 
  $$\prod_{1 \leq i \leq j \leq n} C_{ij} = \prod_{1 \leq i \leq j \leq n} p^{(j-i+1)m_{ij}},$$
  cf.\ \eqref{cij} and also \eqref{jminusi}.   Taking geometric means, we get
  $${\rm Kl}_{p}(\underline{m}, \psi, \psi', w_l) \ll \Big(\prod_{1 \leq i \leq j \leq n} C_{ij}\Big)  \Big(\prod_{ 1 \leq i \leq j \leq n} p^{\frac{1}{8}b_{\ast}(i, j)}\Big) .$$
We now observe that
$$\sum_{1 \leq i \leq j \leq n} (n+1-j)b_1(i, j) = -\sum_{1 \leq i \leq j \leq n} (j-i+1)m_{ij}$$
and 
  $$\sum_{i_0 = 1}^i b_{\ast}(i_0, j) \leq b_1(i, j).$$
  Taken together, this implies
   $$\sum_{1 \leq i \leq j \leq n} n^2 b_{\ast}(i, j)  \leq \sum_{1 \leq i \leq j \leq n} (n+1-j)(n+1-i) b_{\ast}(i, j) \leq -\sum_{1 \leq i \leq j \leq n} (j-i+1)m_{ij},$$
  and so
   $${\rm Kl}_{p}(\underline{m}, \psi, \psi', w_l) \ll \Big(\prod_{1 \leq i \leq j \leq n} C_{ij}\Big)^{1 - \frac{1}{8n^2}}  .$$

\subsection{The element $w_{\ast}$}

The proof of Corollary \ref{cor5} is similar. We apply again a soft argument and  use the ordering
$$(1n) < (1, n-1) < \ldots < (11) < (2, n) < (3, n) < \ldots < (n, n).$$
Analyzing \eqref{partial1}, we see that  $\Sigma_{ij}$ is of the shape \eqref{sigmaij} with
\begin{displaymath}
\begin{split}
 &A_{1n} = \psi'_n\overline{c_{nn}} p^{-m_{1n}}, \quad B_{1n} = \psi_n c_{1(n-1)}p^{-m_{1n}},\\
 &A_{1(n-1)} = \psi'_n p^{m_{nn}- m_{1(n-1)} + m_{1n}} + \psi_n\overline{c_{1n}} p^{-m_{1n}},  \quad B_{1(n-1)} = \psi_{n-1}c_{1(n-2)} p^{-m_{1(n-1)}},\\
 & A_{1 j} = \psi_{j+1}\overline{c_{1(j+1)}} p^{-m_{1( j+1)}} ,\quad B_{1j} =  \psi_j c_{1(j-1)} p^{-m_{1j}},   \quad 1 \leq j \leq n-2,\\
 &A_{2n} = \psi'_1p^{-m_{2n}}, \quad B_{2n} = \psi_2c_{3n}p^{-m_{12}+ m_{11} - m_{2n}},\\
 & A_{in} = \psi_{i-1}\overline{c_{(i-1)n}} p^{-m_{1(i-1)} + m_{1(i-2)}  - m_{(i-1)n}}, \quad  B_{in} = \psi_ic_{(i+1)n} p^{-m_{1i} + m_{1(i-1)} -m_{in}}, \quad 3 \leq i \leq n-1,\\
 & A_{nn} = \psi_{n-1}\overline{c_{(n-1)n}}p^{-m_{1(n-1)} + m_{1(n-2)} - m_{(n-1)n} }, \quad B_{nn} = \psi'_nc_{1n}p^{-m_{1n}}+ \psi_np^{m_{1(n-1)}  - m_{1n} -  m_{nn}}
\end{split}
\end{displaymath}
with the same conventions as explained after \eqref{a1etc}.

Arguing as before based on \eqref{kloosterbound},   we obtain
 $$\sum_{\substack{c_{\nu\mu}\\ (\nu, \mu) \not= (i, j)}} \Big| \sum_{c_{ij}} (...) \Big| \ll\big(\max_{1 \leq j \leq n} |\psi_j|^{-1/2}_p  \big) p^{r_1 + \ldots + r_n -m_{ij}/2 + O(\mu_{ij})}$$ 
 and conclude the proof as in Section \ref{sec52} for some $\delta > 0$. 
 
 We can make this quantitative as in Section \ref{refined}. We have
 \begin{displaymath}
 \begin{split}
 & \sum_{c_{11}}(...) \ll |\psi_1|_p^{-1/2} C_{11} p^{-m_{11}/2}, \\
 & \sum_{c_{1(j-1)}} \Big| \sum_{c_{1j}} (...)\Big| \ll |\psi_j|^{-1/2}_p C_{1(j-1)} C_{1j} p^{-m_{1j}/2}, \quad 2 \leq j \leq n,\\
 & \sum_{c_{(i+1)n}} \Big| \sum_{c_{in}} (...)\Big| \ll |\psi_i|_p^{-1/2} C_{(i+1)n}C_{in} p^{-\delta_{m_{in} \not= 0} (m_{in} + m_{1i} - m_{1(i-1)})/2}, \quad 2 \leq i \leq n-1,\\
 \end{split}
 \end{displaymath}
and by a small variation of Lemma \ref{kloosterbound}, we also have
$$\sum_{c_{1n}} \Big| \sum_{c_{nn}} (...) \Big| \ll |\psi_n|_p^{-1/2} C_{1n} C_{nn} p^{-\delta_{m_{nn} \not= 0} (m_{nn} + m_{1n} - m_{1(n-1)})/2}.$$
We put $b(1, j) = - m_{1j}$, $b(i, n) = -m_{in} - m_{1n} + m_{1(n-1)}$ for $i \geq 2$, and $b_{\ast}(i, j)= \delta_{m_{ij} \not = 0} \min(0, b(i, j))$. We put the $2n+1$ nodes $(i, j)$ with $i = 1$ or $j= n$ into the two classes $\mathcal{C}_1$ with indices of the form $(1, \text{odd})$, $(\text{even}, n)$ and $\mathcal{C}_2$ with indices of the form $(1, \text{even})$, $(\text{odd}, n)$.  Then
 $${\rm Kl}_{p}(\underline{m}, \psi, \psi', w_{\ast}) \ll \Big(\prod_{i, j} C_{ij} \Big)  \min_{\nu = 1, 2} \Big(\prod_{(i, j) \in \mathcal{C}_\nu} p^{\frac{1}{2}b_{\ast}(i, j)}\Big) \leq  \Big(\prod_{i, j} C_{ij} \Big)   \Big(\prod_{i, j} p^{\frac{1}{4}b_{\ast}(i, j)}\Big).$$
We now observe that
$$\sum_{2 \leq i \leq n} (n+1 - i) b_{\ast}(i, n) +  \sum_{1 \leq  j \leq n} n b_{\ast}(1, j) \leq -  \sum_{i, j} (j-i+1)m_{ij},$$
and so
$$ {\rm Kl}_{p}(\underline{m}, \psi, \psi', w_{\ast})  \ll  \Big(\prod_{i, j} C_{ij} \Big)^{1- \frac{1}{4n}}.$$ 

\section{An exact evaluation}



Here we prove Corollary \ref{cor6}. For the vector $r = (1, \ldots, 1)$, the relevant $\underline{m}$ satisfying \eqref{rm} are 
\begin{itemize}
\item $m_{1n} = 1$, $m_{ij} = 0$ \text{ otherwise;}
\item $m_{1k} = m_{k+1, n} = 1$, $m_{ij} = 0$ \text{ otherwise,} for some $1 \leq k \leq n-1$. 
\end{itemize}
In the first case we obtain
$${\rm Kl}_p(\underline{m}, \psi, \psi', w_{\ast}) =\sum_{c_{11}, \ldots, c_{1(n-2)} \, (\text{mod } p)} \sum_{c_{1(n-1)} \text{(mod } p)}\underset{c_{1n} \text{(mod } p)}{\left.\sum\right.^{\ast}} e\Big(\frac{\psi_n c_{1(n-1)} \overline{c_{1n}}}{p} + \frac{\psi'_n c_{1(n-1)}}{p}\Big).$$
The two innermost sums equal $p$, and so we obtain ${\rm Kl}_p(\underline{m}, \psi, \psi', w_{\ast})  = p^{n-1}.$

In the second case we consider first the case $2 \leq k \leq n-1$. Then ${\rm Kl}_p(\underline{m}, \psi, \psi', w_{\ast})$ contains the sum
$$\sum_{\substack{c_{1(k-1)}, c_{1k} \, (\text{mod }p)\\ (c_{1k}, p) = 1}} e\Big(\frac{\psi_k c_{1(k-1)} \overline{c_{1k}}}{p}\Big) = 0.$$
Finally, if $k = 1$, we obtain 
$${\rm Kl}_p(\underline{m}, \psi, \psi', w_{\ast}) =\sum_{c_{3n}, \ldots, c_{nn} \, (\text{mod } p)} \underset{c_{11} ,c_{2n} \text{(mod } p)}{\left.\sum\right.^{\ast}}  e\Big(\frac{\psi_1 \overline{c_{11}} }{p} + \frac{\psi'_1 c_{2 n}}{p}\Big).$$
The two inner sums equal 1, and so ${\rm Kl}_p(\underline{m}, \psi, \psi', w_{\ast})  = p^{n-2}.$ This completes the proof.

\section{Beyond Sarnak's density conjecture}

We finally prove Proposition \ref{prop1}. This requires some minor modifications in Sections 4 and 5 of \cite{AB} that we now describe. We use the notation from \cite{AB}. Since $q$ is prime, the argument simplifies a bit, and we need \cite[Lemma 4.2]{AB} for $(\alpha, \beta) = (0, 0)$ and $(1, 0)$. For $(\alpha,  \beta) = (0, 0)$ we use it as is, for $(\alpha, \beta) = (1, 0)$ we make a small improvement. We recall that we need to count $x'_{ij}, y'_{ij}$ satisfying the size conditions \cite[(4.15)]{AB} and the congruences \cite[(4.17), (4.18)]{AB}, where in the case $\beta = 0$ the congruence (4.18) can be written more simply as (4.19). The count for (4.15) is given in (4.16). In order to count the saving imposed by the congruences (4.17), (4.19), we proceed as described after  (4.18), but obtain an extra saving for $y_{1n}'$ from (4.19) of size $p^{n-3}$. Thus we see that
\begin{displaymath}
\begin{split}
 C_{1, 0} & \leq \frac{p^{\frac{1}{6}(n^3 + 3n^2 + 2n - 12) +  n(n-1)  }}{p^{2(n-2)} \cdot p^{\frac{1}{2}(n-2)(n-3)} \cdot p^{n-2 + \frac{1}{2}(n-1)(n-2)   } \cdot p^{n-3}}
 = \mathcal{N}_q q^{(n-1)^2} q^{n+2} \leq \mathcal{N}_q q^{(n-1)^2} q^{\frac{7}{4}(n-1)}
 \end{split}
 \end{displaymath}
for $n  \geq 5$. (It is important to have exponent strictly less than 2.) Together with the improved bound of Corollary \ref{cor5}, we  now obtain the following variation of \cite[Theorem 4.3]{AB} under the additional assumptions that $n \geq 5$, $q$ is prime and $c = (c_1, \ldots, c_{n-1}) = (q^n \gamma_1, \ldots, q^n\gamma_{n-1})$ satisfies $\gamma_j < q^{2}$ (which implies that only the cases $(\alpha, \beta) = (0, 0)$ and $(1, 0)$ are relevant in the proof):
\begin{equation}\label{new}
S_{q, w_{\ast}}^v(M, N, c) \ll q^{\varepsilon} \frac{\mathcal{N}_q}{q^{n-1}}   \frac{c_1\cdot \ldots \cdot c_{n-1}}{(\gamma_1 \cdot\ldots \cdot \gamma_{n-1})^{\delta}} \big(\gamma_1 \cdots \gamma_{n-1} , q^{\infty}\big)^{3/4 + \delta}
\end{equation}
 with $\delta$ as in our Corollary \ref{cor5}. We can and will assume without loss of generality that $\delta < 1/10$. 

With this in hand, we move to the discussion after \cite[Lemma 5.1]{AB}. The key point is that we can now slightly relax \cite[(5.3)]{AB} to 
$$mZ \leq K^{-1} (1 + 1/r + R)^{-K} q^{n+1 + \delta_0}$$
for some sufficiently small $\delta_0 > 0$ to be chosen in a moment (cf.\ also \cite[(1.4)]{AB}). If $n \geq 4$ and $\delta_0 < 1/10$, then by Remark 2 after \cite[Lemma 4.1]{AB} we can still conclude that only the trivial Weyl element and $w_{\ast}$ give a non-zero contribution. The contribution of the trivial Weyl element is given in \cite[(5.4)]{AB}, for the contribution of the $w_{\ast}$ we invoke \eqref{new}
getting
$$\preccurlyeq q^{\varepsilon} \frac{\mathcal{N}_q}{q^{n-1}} Z^{2\eta_1} \sum_{ \gamma_1, \ldots, \gamma_{n-1} \ll q^{1+\delta_0}} \frac{(\gamma_1 \cdots \gamma_{n-1}, q^{\infty})}{(\gamma_1\cdots \gamma_{n-1})^{\delta}} \preccurlyeq \mathcal{N}_q Z^{2\eta_1} $$
if $(1 + \delta_0)(1 - \delta) < 1$. In this way we obtain an improved version of \cite[Proposition 5.2]{AB} where (under the current assumptions $q$ prime, $n \geq 5$) we only need the relaxed condition $T \leq M^{-K} q^{n+1 + \delta_0}$. This can be directly inserted into \cite[Corollary 6.11]{AB} and completes the proof of Proposition \ref{prop1}.

\section{Appendix: the case $G = {\rm GL}(4)$}

It might be useful for applications to use the method of proof of Theorem \ref{thm1} and Corollary~\ref{cor3} to obtain explicit formulae and non-trivial bounds for \emph{all} Weyl elements in the case $G = {\rm GL}(4)$. See \cite[Appendix]{GSW} for a list of the relevant consistency relations, and \cite{Fr0} for a version in terms of Pl\"ucker coordinates. 

There are 8 Weyl elements. We do not need to talk about the trivial Weyl element and the  Voronoi element $\left(\begin{smallmatrix} & 1\\ I_{3} & \end{smallmatrix}\right)$, which is covered in \cite{Fr} (with non-trivial bounds following from Deligne's estimates). The element $\left(\begin{smallmatrix} &   I_{3} \\ 1& \end{smallmatrix}\right)$ is analogous. 

\bigskip

(a) For the long Weyl element $w_l = \left(\begin{smallmatrix}&&&-1\\&&1\\&-1\\1\end{smallmatrix}\right)$ we obtain by \eqref{partial} that ${\rm Kl}_{p}(\underline{m}, \psi, \psi', w_l)$ is given by
{\footnotesize \begin{equation}\label{81}
\begin{split}
\sum_{\underline{c} \in \mathcal{C}_{w_l}(\underline{m})} &e\Bigg(\frac{\psi_1 \overline{c_{11}}}{p^{m_{11}}} + \psi_2\Big(\frac{c_{11} \overline{c_{12}}}{p^{m_{12}}} + \frac{\overline{c_{22}}}{p^{-m_{11} + m_{12} + m_{22}}}\Big) + \psi_3\Big(\frac{c_{12} \overline{c_{13}}}{p^{m_{13}}} + \frac{c_{22} \overline{c_{23}}}{p^{-m_{12} + m_{13}+ m_{23}}} + \frac{\overline{c_{33}}}{p^{-m_{12} - m_{22} + m_{13} + m_{23} + m_{33}}}\Big)\\
& + \frac{\psi'_{1}c_{33}}{p^{m_{33}}} + \psi'_2\Big(\frac{c_{23}\overline{c_{33}}}{p^{m_{23}}} + \frac{c_{22}}{p^{-m_{33}  + m_{23} + m_{22}}}\Big) + \psi'_3\Big(\frac{c_{13} \overline{c_{23}}}{p^{m_{13}}} + \frac{c_{12} \overline{c_{22}}}{p^{-m_{23} + m_{13} + m_{12}}} + \frac{c_{11}}{p^{-m_{23} - m_{22} + m_{13} + m_{12} + m_{11}}}\Big)\Bigg)
\end{split}
\end{equation}}
where the sum runs over
\begin{displaymath}
\begin{split}
& c_{11} \, (\text{mod } p^{m_{11} + m_{12} + m_{13}}), \quad c_{12} \, (\text{mod } p^{  m_{12} + m_{13}}), \quad c_{13} \, (\text{mod } p^{ m_{13}}),\\
&c_{22} \, (\text{mod } p^{m_{22} + m_{23 }}), \quad c_{23} \, (\text{mod } p^{  m_{23}}), \quad c_{33} \, (\text{mod } p^{ m_{33}})
\end{split}
\end{displaymath}
subject to $(c_{ij}, p^{m_{ij}}) = 1$.

\bigskip

(b) For the  Weyl element $w_{\ast} = \left(\begin{smallmatrix}&&&-1\\&-1\\&&-1\\1\end{smallmatrix}\right)$ we obtain by \eqref{partial1} that ${\rm Kl}_{p}(\underline{m}, \psi, \psi', w_{\ast})$ is given by
{\footnotesize \begin{equation}\label{82}
\begin{split}
\sum_{\underline{c} \in \mathcal{C}_{w_{\ast}}(\underline{m})} &e\Bigg(\frac{\psi_1 \overline{c_{11}}}{p^{m_{11}}} + \psi_2\Big(\frac{c_{11} \overline{c_{12}}}{p^{m_{12}}} + \frac{c_{33}\overline{c_{23}}}{p^{-m_{11} + m_{12} + m_{22}}}\Big) + \psi_3\Big(\frac{c_{12} \overline{c_{13}}}{p^{m_{13}}} + \frac{  \overline{c_{33}}}{p^{-m_{12} + m_{13}+ m_{33}}} \Big)\\
& + \frac{\psi'_{1}c_{23}}{p^{m_{23}}}  + \psi'_3\Big(\frac{c_{13} \overline{c_{33}}}{p^{m_{13}}} + \frac{c_{12} }{p^{-m_{33} + m_{13} + m_{12}}} \Big)\Bigg)
\end{split}
\end{equation}}
where the sum runs over
\begin{displaymath}
\begin{split}
& c_{11} \, (\text{mod } p^{m_{11} + m_{12} + m_{13}}), \quad c_{12} \, (\text{mod } p^{  m_{12} + m_{13}}), \quad c_{13} \, (\text{mod } p^{ m_{13}}),\\
&  c_{23} \, (\text{mod } p^{    m_{33}}), \quad c_{33} \, (\text{mod } p^{ m_{23} + m_{33}})
\end{split}
\end{displaymath}
subject to $(c_{ij}, p^{m_{ij}}) = 1$. These two cases are illustrative examples for the general results presented in the body of the paper;  Corollaries \ref{cor3} and \ref{cor5} establish non-trivial bounds. 

\bigskip

(c) We are left with three remaining Weyl elements. We first consider $w = \left(\begin{smallmatrix}&&1\\&&&1\\1&&\\&1\end{smallmatrix}\right)$. Here we choose the representative $w = s_{\alpha_2} s_{\alpha_1} s_{\alpha_3} s_{\alpha_2}$, so that
 \ba
\ul\gamma = \rb{\alpha_{22},\alpha_{12},\alpha_{23},\alpha_{13}}.
\ea
We then obtain for ${\rm Kl}_{p}(\underline{m}, \psi, \psi', w )$   by similar computations the formula
{\footnotesize \begin{displaymath}
\begin{split}
\underset{\substack{c_{22} \, (\text{mod } p^{m_{12} + m_{22} + m_{23}})\\ c_{12} \, (\text{mod } p^{m_{12} + m_{13}})\\ c_{23} \, (\text{mod } p^{m_{13} + m_{23}})\\ c_{13} \, (\text{mod } p^{m_{13}})}}{\left.\sum \right.^{\ast}} &e\Bigg(\psi_1 \Big(\frac{c_{22}\overline{c_{12}}}{p^{m_{12}}}  + \frac{c_{23} \overline{c_{13}}}{p^{-m_{22} + m_{12} + m_{13}}}\Big) + \frac{\psi_2\overline{c_{22}}}{p^{m_{22}}}  + \psi_3\Big(\frac{c_{22}\overline{c_{23}}}{p^{m_{23}}} + \frac{c_{12}\overline{c_{13}}}{p^{-m_{22} + m_{13} + m_{23}}}\Big) + \frac{\psi_2'c_{13}}{p^{m_{13}}}\Bigg)
\end{split}
\end{displaymath}}
where $\sum^{\ast}$ indicates the usual coprimality condition $(c_{ij}, p^{m_{ij}}) = 1$.

 Here we obtain a saving relative to the trivial bound as follows: the $c_{13}$-sum and the $c_{22}$-sum save $O ( |\psi'_2|^{-1/2}_p p^{-m_{13}/2})$ and $O( |\psi_2|^{-1/2}_p p^{-m_{22}/2})$ respectively. If $p \nmid c_{22}$, an analogous argument works for the indices $(12), (23)$, and if $m_{22} = 0$ (which is the only situation in which we can have $p \mid c_{22}$), then the $c_{22}$-sum implies 
 $$\psi_1 \overline{c_{12}} p^{m_{23}} + \psi_3 \overline{c_{23}} p^{m_{12}} \equiv 0 \, (\text{mod } p^{m_{12} + m_{23}}).$$
 Thus in all cases we get a saving of 
 $$ \max(|\psi_1|_p^{-1/2}, |\psi_2|_p^{-1/2}, |\psi_3|_p^{-1/2}, |\psi_1'|_p^{-1/2}) p^{-\frac{1}{2}\max(m_{12}, m_{13}, m_{22}, m_{23})},$$
 so that one can choose $\delta <  1/16$ in Corollary \ref{cor5} below. This is just for concreteness -- it is easy to improve the numerical value.


\bigskip

(d) Finally we treat the Weyl element $w = \left(\begin{smallmatrix}&& & -1\\&&1&\\1&&\\&1\end{smallmatrix}\right)$, the Weyl element $  \left(\begin{smallmatrix}&&1 & \\& & &1&\\&-1&\\1&\end{smallmatrix}\right)$ being analogous. Here we choose a representative $w = s_{\alpha_1} s_{\alpha_2} s_{\alpha_3} s_{\alpha_1} s_{\alpha_2}$, so that \ba
\ul\gamma = \rb{\alpha_{11}, \alpha_{12}, \alpha_{13}, \alpha_{22}, \alpha_{23}}.
\ea
We then obtain for ${\rm Kl}_{p}(\underline{m}, \psi, \psi', w )$    the formula
{\footnotesize \begin{displaymath}
\begin{split}
\underset{\substack{c_{11} \, (\text{mod } p^{m_{11} + m_{12} + m_{13}})\\ c_{12} \, (\text{mod } p^{m_{12} +  m_{13}})\\ c_{13} \, (\text{mod } p^{  m_{13}})\\ c_{22} \, (\text{mod } p^{m_{22} + m_{23}})\\ c_{23} \, (\text{mod } p^{m_{23}})}}{\left.\sum \right.^{\ast}} &e\Bigg( \frac{\psi_1\overline{c_{11}}}{p^{m_{11}}} +  \psi_2\Big(\frac{c_{11}\overline{c_{12}}}{p^{m_{12}}} + \frac{\overline{c_{22}}}{p^{-m_{11} + m_{12} + m_{22}}}\Big) + \psi_3\Big(\frac{c_{12} \overline{c_{13}}}{p^{m_{13}}} + \frac{c_{22} \overline{c_{23}}}{p^{-m_{12} + m_{13} + m_{23}}}\Big) \\
&+ \frac{\psi_2'c_{23}}{p^{m_{23}}} + \psi_3' \Big(\frac{c_{11}}{p^{-m_{22}-m_{23} + m_{11} + m_{12} + m_{13}}}+ \frac{c_{12} \overline{c_{22}}}{p^{-m_{23} + m_{12} + m_{13}}} + \frac{c_{13} \overline{c_{23}}}{p^{m_{13}}}\Big)\Bigg).
\end{split}
\end{displaymath}}
Arguing as in Section \ref{non-trivial} with the ordering
$$(13) < (12) < (11) < (23) < (22)$$
we obtain a saving of 
$$\ll_{\psi, \psi'} p ^{-\frac{1}{2}\max( m_{13}, m_{12}, m_{11}, m_{23}, m_{22} - m_{11})} \ll p ^{-\frac{1}{4}\max( m_{13}, m_{12}, m_{11}, m_{23}, m_{22} )}$$
so that we can choose $\delta < 1/36$ in the following corollary. Again it is very easy to improve this numerical value. 
We conclude

\begin{cor}
Let $w$ be a non-trivial Weyl element for $G = {\rm GL}(4)$ and $C = (p^{r_1}, p^{r_2}, p^{r_3})$. Then there exists an absolute constant $\delta  > 0$ such that 
$${\rm Kl}_{p}(  \psi, \psi', C^{\ast}w)  \ll \big(\max_{1 \leq j \leq n} (|\psi_j|^{-1/2}_p ,  |\psi'_j|^{-1/2}_p ) \big) \big( p^{r_1+r_2+r_3}\big)^{1-\delta}.$$

\end{cor}

\vspace{0.4cm}

\noindent \textbf{Conflicts of interest.} The authors declare that   there is no conflict of interest.\\
\textbf{Availability of data and material.} Not applicable.

\end{document}